\documentclass{article}

\usepackage{amssymb}
\usepackage{amsfonts}
\usepackage{graphicx}
\usepackage{amsmath}
\usepackage[T1]{fontenc}

\newtheorem{theorem}{Théorème}[section]

\newtheorem{corollary}[theorem]{Corollaire}

\newtheorem{definition}[theorem]{Définition}

\newtheorem{proposition}[theorem]{Proposition}

\newenvironment{proof}[1][Preuve]{\textbf{#1.} }{\ \rule{0.5em}{0.5em}}

\begin{document}

\title{De la caract\'{e}risation matricielle des drapeaux complets
d'extensions riemanniennes aux feuilletages riemanniens transversalement
diagonaux sur une vari\'{e}t\'{e} compacte connexe}
\author{\textit{Cyrille Dadi}$^{(1)}$\textit{\ \ et Adolphe Codjia}$^{\left(
2\right) }$ \\
Laboratoire de Math\'{e}matiques Fondamentales, \\
Universit\'{e} F\'{e}lix Houphou\"{e}t-Boigny, ENS\\
\ 08 BP 10 Abidjan (C\^{o}te d'Ivoire)\\
\textit{\ email}$^{(1)}:$\textit{cyriledadi@yahoo.fr , email}$^{(2)}:$%
\textit{ad\_wolf2000@yahoo.fr}\\
\ }
\maketitle

\bigskip

{\Large R\'{e}sum\'{e}}

Ici, nous caract\'{e}risons matriciellement les drapeaux complets
d'extensions riemanniennes 
$\mathcal{D}_{_{\mathcal{F}_{q}\text{ }}}=\left( \mathcal{F}_{q-1},\text{ }\mathcal{F}_{q-2},\text{ }...,\text{ }\mathcal{F}_{1}\right)$
d'un feuilletage riemannien $\mathcal{F}_{q}$ sur une vari\'{e}%
t\'{e} compacte connexe sous l'hypoth\`{e}se d'existense d'une m\'{e}trique
quasi-fibr\'{e}e commune \`{a} tous les feuilletages $\mathcal{F}_{s}$ de ce
drapeau. Cette caract\'{e}risation nous permet de voir que chaque
feuilletage de ce drapeau est un feuilletage riemannien transversalement
diagonal.

\textbf{Mots-cl\'{e}s}: \textit{feuilletage riemannien transversalement
diagonal, extension d'un feuilletage}\textbf{, }drapeau complet d'extension
riemannienne\textit{.}

\bigskip

{\Large Abstract}

In this paper we characterise with the matrix the complete flag of
riemannian\ extension (see d\'{e}finition) $\mathcal{D}_{_{\mathcal{F}_{q}%
\text{ }}}=\left( \mathcal{F}_{q-1},\text{ }\mathcal{F}_{q-2},\text{ }...,%
\text{ }\mathcal{F}_{1}\right) $ on a riemannian compact manifold whose
metric is bundlelike for any foliation $\mathcal{F}_{s}$ of this flag.

This study show us that a foliation of a complete flag of riemannian
extension on a riemannian compact manifold whose metric is bundlelike for
any foliation of this flag is a transversely diagonal riemannian foliation
(see d\'{e}finition).\textit{\ }

\textbf{Keywords}:\textit{\ transversely diagonal riemannian foliation,
extension of a foliation, complete flag of riemannian\ extension.}

\section{Introduction}

Etant donn\'{e} un feuilletage $\mathcal{F}_{q}$ de codimension $q$ sur une
vari\'{e}t\'{e} $M,$ un drapeau d'extension du feuilletage $\mathcal{F}_{q}$
est une suite $\mathcal{D}_{_{\mathcal{F}_{q}\text{ }}}^{k}=\left( \mathcal{F%
}_{q-1},\text{ }\mathcal{F}_{q-2},\text{ }...,\text{ }\mathcal{F}_{k}\right) 
$ de feuilletages sur la vari\'{e}t\'{e} $M$ telle que $\mathcal{F}%
_{q}\subset \mathcal{F}_{q-1}\subset \mathcal{F}_{q-2}\subset $ $...\subset 
\mathcal{F}_{k}$ et chaque feuilletage $\mathcal{F}_{s}$ est de codimension $%
s.$

Pour $k=1$, le drapeau d'extension $\mathcal{D}_{_{\mathcal{F}_{q}\text{ }%
}}^{k}$ sera dit complet et sera not\'{e} $\mathcal{D}_{_{\mathcal{F}_{q}%
\text{ }}}.$

Si chaque feuilletage $\mathcal{F}_{s}$ est riemannien, le drapeau
d'extension $\mathcal{D}_{_{\mathcal{F}_{q}\text{ }}}^{k}$ sera appel\'{e}
drapeau d'extension riemannienne du feuilletage $\mathcal{F}_{q}.$

On montre que si $\mathcal{F}_{q}$ est un feuilletage riemannien de
codimension $q$ admettant un drapeau complet d'extension riemannienne $%
\mathcal{D}_{_{\mathcal{F}_{q}\text{ }}}=\left( \mathcal{F}_{q-1},\text{ }%
\mathcal{F}_{q-2},\text{ }...,\text{ }\mathcal{F}_{1}\right) $ sur une vari%
\'{e}t\'{e} compacte connexe $M$ et s'il existe une m\'{e}trique $\mathcal{F}%
_{s}$-quasifibr\'{e}e commune \`{a} tous les feuilletages $\mathcal{F}_{s}$
\ alors chaque feuilletages $\mathcal{F}_{s}$ est transversalement diagonal
c'est \`{a} dire que chaque feuilletage $\mathcal{F}_{s}$ est d\'{e}fini par
un cocycle feuillet\'{e} $(U_{_{i}},f_{_{i}},T,\gamma _{_{ij}})_{i\in I}$ v%
\'{e}rifiant les conditions suivantes:

\textit{i)} les ouverts $U_{_{i}}$ sont $\mathcal{F}$-distingu\'{e}s,

\textit{ii)} la m\'{e}trique $g_{_{T}}$ est une m\'{e}trique sur la vari\'{e}%
t\'{e} transverse $T$ et les $\gamma _{_{ij}}$ sont des isom\'{e}tries
locales pour cette m\'{e}trique,

\textit{iii)} sur chaque ouvert $f_{_{i}}\left( U_{_{i}}\right) $ il existe
un syst\`{e}me $\left( y_{1,}^{i}y_{2,...,}^{i}y_{q}^{i}\right) $ de coordonn%
\'{e}s $\mathcal{F}$-transverses locales tel que relativement aux syst\`{e}%
mes de coordonn\'{e}s $\mathcal{F}$-transverses locales $\left( \left(
y_{1,}^{i}y_{2,...,}^{i}y_{q}^{i}\right) \right) _{i\in I}$ sur les ouverts $%
f_{_{i}}\left( U_{_{i}}\right) $, la matrice jacobienne $J_{\gamma _{_{ij}}}$
de $\gamma _{_{ij}}$ est diagonale et orthogonale.

On \'{e}tablit aussi pour un feuilletage riemannien minimal $\left( M,%
\mathcal{F}_{q},g_{_{T}}\right) $ de codimension $q$ sur une vari\'{e}t\'{e}
compacte connexe $M$ d'alg\`{e}bre de Lie structurale $\mathcal{G}$\ ce qui
suit:

1) Il existe une sous-alg\`{e}bre de Lie $\mathcal{G}_{_{\mathcal{F}}}$ de $%
\mathcal{G}$ telle qu'\`{a} toute extension $\mathcal{F}^{\prime }$ de $%
\mathcal{F}_{q}$ correspond (de fa\c{c}on unique) une sous-alg\`{e}bre de
Lie $\mathcal{G}^{\prime }$ de $\mathcal{G}$\ v\'{e}rifiant $\mathcal{G}_{_{%
\mathcal{F}}}\subset \mathcal{G}^{\prime }$.

2) Toute extension $\mathcal{F}^{\prime }$ de $\mathcal{F}_{q}$ est un
feuilletage riemannien et il existe une m\'{e}trique quasi-fibr\'{e}e
commune \`{a} $\mathcal{F}_{q}$ et $\mathcal{F}^{\prime }.$

3) $\mathcal{F}_{q}$ admet un drapeau complet d'extension si et seulement si 
$\mathcal{F}_{q}$ est transversalement diagonal.

Dans tout ce qui suit, les vari\'{e}t\'{e}s consid\'{e}r\'{e}es sont suppos%
\'{e}es connexes et la diff\'{e}rentiabilit\'{e} C$^{\infty }$.

\section{Rappels}

Dans ce paragraphe, nous r\'{e}formulons dans le sens qui nous est utile
certaines d\'{e}finitions et certains th\'{e}or\`{e}mes qui se trouvent dans
($\cite{BH}$, $\cite{DAD},$ $\cite{DD},$ $\cite{DIA2},\cite{FED}$ $\cite{GOD}%
,$ $\cite{MOL},\cite{MAL}$ $\cite{MAS},$ $\cite{SA}).$

\begin{theorem}
$\cite{MOL}$ Soient $\mathcal{F}$ un feuilletage riemannien de codimension $%
q $ sur une vari\'{e}t\'{e} compacte connexe $M,$ $\overline{F^{\natural }}$
l'adh\'{e}rence d'une feuille $F^{\natural }$ du feuilletage relev\'{e} $%
\mathcal{F}^{\natural }$ de $\mathcal{F}$ sur le\textit{\ }fibr\'{e} $\pi $:$%
M^{\natural }\rightarrow M$\textit{\ }des\textit{\ }rep\`{e}res\textit{\ }%
orthonorm\'{e}s\textit{\ }$\mathcal{F}$\textit{-}transverses,\newline
$j:\overline{F^{\natural }}\hookrightarrow M^{\natural }$ l'injection
canonique de $\overline{F^{\natural }}$ dans $M^{\natural }$, $\phi =\pi
\circ j$ et $z\in \overline{F^{\natural }}.$

Aors:

i) $\pi \left( F^{\natural }\right) $ est une feuille de $\mathcal{F}$ et $%
\phi \left( \overline{F^{\natural }}\right) =\overline{\pi \left(
F^{\natural }\right) }.$

ii) L'application $\phi :\overline{F^{\natural }}\rightarrow \phi \left( 
\overline{F^{\natural }}\right) $ est une fibration principale.
\end{theorem}

\begin{definition}
Une extension d$^{\prime }$un feuilletage $\left( M,\mathcal{F}\right) $ de
codimension $q$ est un feuilletage $\left( M,\mathcal{F}^{\prime }\right) $
de codimension $q^{\prime }$ tel que $0<q^{\prime }<q$ et les feuilles de $%
\left( M,\mathcal{F}^{\prime }\right) $ sont des r\'{e}unions de feuilles de 
$\left( M,\mathcal{F}\right) \ \left( \text{on notera }\mathcal{F}\subset 
\mathcal{F}^{\prime }\right) $.
\end{definition}

On montre que si $\left( M,\mathcal{F}^{\prime }\right) $ est une extension
simple d'un feuilletage simple $\left( M,\mathcal{F}\right) $ et si $\left(
M,\mathcal{F}\right) $ et $\left( M,\mathcal{F}^{\prime }\right) $ sont d%
\'{e}finis respectivement par les submersions $\pi :M\rightarrow T$ et $\pi
^{\prime }:M\rightarrow T^{\prime },$ alors il existe une submersion $\alpha
:T\rightarrow T^{\prime }$ telle que $\pi ^{\prime }=\alpha \circ \pi .$

On dira que la submersion $\alpha $ est une liaison\textbf{\ }entre le
feuilletage $\left( M,\mathcal{F}\right) $ et son feuilletage extension $%
\left( M,\mathcal{F}^{\prime }\right) .$

On montre dans $\cite{DAD}$\ que si le feuilletage $\left( M,\mathcal{F}%
\right) $ et son extension $\left( M,\mathcal{F}^{\prime }\right) $ sont d%
\'{e}finis respectivement par les cocycles $(U_{_{i}},f_{_{i}},T,\gamma
_{_{ij}})_{i\in I}$ et $(U_{_{i}},f_{_{i}}^{\prime },T^{\prime },\gamma
_{_{ij}}^{\prime })_{i\in I}$ alors on a%
\begin{equation*}
f_{i}^{\prime }=\theta _{i}\circ f_{i}\text{ \ et }\gamma _{_{ij}}^{\prime
}\circ \theta _{j}=\theta _{i}\circ \gamma _{_{ij}}
\end{equation*}%
o\`{u} $\theta _{s}$\ est une liaison entre le feuilletage $\left( U_{s},%
\mathcal{F}\right) $ et son feuilletage extension $\left( U_{s},\mathcal{F}%
^{\prime }\right) .$

\begin{proposition}
$\cite{DD}$ Etant donn\'{e} un feuilletage $(M$,$\mathcal{F)}$ de vari\'{e}t%
\'{e} transverse mod\`{e}le $T$, soit $T^{\prime }$ une vari\'{e}t\'{e} de
dimension q'~>~0.\ Si les diff\'{e}omorphismes locaux de
transition \ de $\mathcal{F}$ pr\'{e}servent les fibres d$^{\prime }$une
submersion de $T$ sur $T^{\prime },$ alors le feuilletage $\mathcal{F}$
admet une extension de codimension q' et de vari\'{e}t\'{e} mod\`{e}le
transverse $T^{\prime }.$
\end{proposition}

\begin{theorem}
$\cite{BH}$ Soit $\mathcal{F}$ un feuilletage riemannien transversalement int%
\'{e}grable\ sur une vari\'{e}t\'{e} riemannienne connexe compl\`{e}te $%
\left( M,g\right) $, $\widetilde{M}$ le rev\`{e}tement universel de $\left(
M,g\right) $, $\widetilde{\mathcal{F}}$ et $\widetilde{\mathcal{F}}^{\bot }$%
\ les feuilletages relev\'{e}s respectifs de $\mathcal{F}$ et $\mathcal{F}%
^{\bot }$ sur $\widetilde{M}.$

Alors:

i) $\widetilde{M}$ est diff\'{e}omorphe \`{a} $\widetilde{L}\times 
\widetilde{H}$ o\`{u} $\widetilde{L}$ et $\widetilde{H}$ sont respectivement
des feuilles de $\widetilde{\mathcal{F}}$ et $\widetilde{\mathcal{F}}^{\bot
}.$

ii) $\widetilde{L}$ et $\widetilde{H}$ sont des rev\^{e}tements universels
respectifs des feuilles respectives $L$ et $H$ des feuilletages respectifs $%
\mathcal{F}$ et $\mathcal{F}^{\bot }.$

iii) Toute feuille de $\widetilde{\mathcal{F}}$ est diff\'{e}omorphe \`{a} $%
\widetilde{L}\times \left\{ p\right\} $ o\`{u} $p\in \widetilde{H}.$

iv) Toute feuille de $\widetilde{\mathcal{F}}^{\bot }$ est diff\'{e}omorphe 
\`{a} $\left\{ p\right\} \times \widetilde{H}$ o\`{u} $p\in \widetilde{L}.$

v) La projection de $\widetilde{M}$ sur $\widetilde{H}$ suivant $\widetilde{L%
}$ est une submersion riemannienne.
\end{theorem}

\begin{theorem}
$\cite{DIA2}$ Soit $\mathcal{D}=\left( \mathcal{F}_{1},\text{ }\mathcal{F}%
_{2},\text{ }...,\text{ }\mathcal{F}_{n-1}\right) $ un drapeau d'une $n$-vari%
\'{e}t\'{e} riemannienne connexe compacte $\left( M,g\right) .$ Si la m\'{e}%
trique $g$ est quasi-fibr\'{e}e pour chacun des feuilletages, alors on a les
propri\'{e}t\'{e}s suivantes:

1) La vari\'{e}t\'{e} $M$ est compl\`{e}tement parall\'{e}lisable et les
feuilletages du drapeau sont \`{a} la fois transversalement parall\'{e}%
lisables et transversalement int\'{e}grables. De fa\c{c}on pr\'{e}cise, il
existe un unique parall\'{e}lisme $\left( Y_{k}\right) _{1\leq k\leq n\text{
\ \ }}$de $M$ dit parall\'{e}lisme du drapeau tel que:

- pour tout $k\geq 1,$ $Y_{k}$ est unitaire, tangent \`{a} $\mathcal{F}_{k}$
et oriente le flot qu'il d\'{e}finit,

- pour tout $1\leq k$\guilsinglleft $s\leq n,$ $\left[ Y_{k},Y_{s}\right]
=c_{_{ks}}Y_{k}$ o\`{u} les fonctions $c_{_{ks}}$ , dites fonctions de
structures du drapeau, sont, pour $k\geq 2,$ basiques pour le feuilletage $%
\mathcal{F}_{k-1.}$

2) Le drapeau se rel\`{e}ve sur le rev\`{e}tement universel $\widetilde{M}$
de $M$\ en un drapeau $\widetilde{\mathcal{D}}=\left( \widetilde{\mathcal{F}}%
_{1},\text{ }\widetilde{\mathcal{F}}_{2},\text{ }...,\text{ }\widetilde{%
\mathcal{F}}_{n-1}\right) $ de feuilletages simple. De plus pour tout $k\in
\left\{ 1,2,...,n-1\right\} $ le rev\`{e}tement universel $\widetilde{M}$
est diff\'{e}omorphe \`{a} $\widetilde{F}_{k}\times \widetilde{F}_{k}^{\bot
} $ o\`{u} $\widetilde{F}_{k}$ et $\widetilde{F}_{k}^{\bot }$ sont
respectivement des feuilles de $\widetilde{\mathcal{F}}_{k}$ et $\widetilde{%
\mathcal{F}}_{k}^{\bot }.$
\end{theorem}

On notera que :

\textit{i)} Les champs de vecteurs $Y_{k}$ du parall\'{e}lisme $\left(
Y_{k}\right) _{1\leq k\leq n\text{ \ \ }}$de $M$\ sont orthogonaux deux \`{a}
deux.

\textit{ii)} Sur chaque feuille $F_{s}^{\bot },$ le parall\'{e}lisme $\left(
Y_{k}\right) _{1\leq k\leq n\text{ \ \ }}$induit un parall\'{e}lisme que
l'on notera $\left( Y_{k}\right) _{s\leq k\leq n\text{ \ \ }}.$

\textit{iii)} Le parall\'{e}lisme $\left( Y_{k}\right) _{s\leq k\leq n\text{
\ \ }}$ de $F_{s}^{\bot }$ se rel\`{e}ve sur le rev\^{e}tement universel $%
\widetilde{M}$ de $M$ en un parral\'{e}lisme $\left( \widetilde{Y}%
_{k}\right) _{s\leq k\leq n\text{ }}$ sur $\widetilde{F}_{s}^{\bot }$.

\begin{theorem}
$\cite{DD}$ Soit $\left( M,\mathcal{F}\right) $ un $G$-feuilletage de Lie
minimal d'une vari\'{e}t\'{e} compacte et connexe, d'alg\`{e}bre de Lie $%
\mathcal{G}.$

Alors:

1-Il y a une correspondance biunivoque entre les sous-alg\`{e}bres de Lie de%
\textit{\ }$\mathcal{G=}Lie\left( G\right) $ $($ou si l'on pr\'{e}f\`{e}re
entre les sous-groupes de Lie connexes de\textit{\ }$G)$ et les extensions
de $\mathcal{F}.$

2- Une extension de $\mathcal{F}$ est un $\frac{\mathcal{G}}{\mathcal{H}}$%
-feuilletage transversalement riemannien \`{a} fibr\'{e} normal trivial, d%
\'{e}finie par une 1-forme vectorielle \`{a} valeurs dans\textit{\ }$\frac{%
\mathcal{G}}{\mathcal{H}}$.

3- Une extension de $\mathcal{F}$ est transversalement homog\`{e}ne $\left( 
\text{resp}.\text{ de Lie}\right) $ si et seulement si le sous-groupe de Lie
de $G$ correspondant est un sous-groupe ferm\'{e} (resp. un sous-groupe
normal) dans $G.$
\end{theorem}

\begin{theorem}
$\cite{RM}$ Soit $M\times N$ le produit de deux vari\'{e}t\'{e}s $M$ et $N$, 
\begin{equation*}
\begin{array}{cccc}
R_{h_{1}}: & T_{x}M & \rightarrow & T_{\left( x,y\right) }M\times N \\ 
& u & \mapsto & u^{h_{1}}=\left( u,0\right)%
\end{array}%
,\text{ }%
\begin{array}{cccc}
R_{h_{2}}: & T_{x}N & \rightarrow & T_{\left( x,y\right) }M\times N \\ 
& w & \mapsto & w^{h_{2}}=\left( 0,w\right)%
\end{array}%
,
\end{equation*}%
$X_{i}\in \mathcal{X(}M\mathcal{)}$ et $Y_{j}\in \mathcal{X(}N\mathcal{)}$.

Alors $\left[ X_{1}^{h_{1}},X_{2}^{h_{1}}\right] =\left[ X_{1},X_{2}\right]
^{h_{1}}$, $\left[ Y_{1}^{h_{2}},Y_{2}^{h_{2}}\right] =\left[ Y_{1},Y_{2}%
\right] ^{h_{2}}$ et $\left[ X_{1}^{h_{1}},Y_{2}^{h_{2}}\right] =0.$
\end{theorem}

\section{Drapeau complet d'extension de feuilletage riemannien sur une vari%
\'{e}t\'{e} compacte connexe}

\begin{definition}
Soit $\mathcal{F}_{q}$ un feuilletage de codimension $q$ sur une vari\'{e}t%
\'{e} $M.$

Un drapeau d'extension du feuilletage $\mathcal{F}_{q}$ est une suite 
\newline
$\mathcal{D}_{_{\mathcal{F}_{q}\text{ }}}^{k}=\left( \mathcal{F}_{q-1},\text{
}\mathcal{F}_{q-2},\text{ }...,\text{ }\mathcal{F}_{k}\right) $ de
feuilletages sur la vari\'{e}t\'{e} $M$ telle que\newline
$\mathcal{F}_{q}\subset \mathcal{F}_{q-1}\subset \mathcal{F}_{q-2}\subset $ $%
...\subset \mathcal{F}_{k}$ et chaque feuilletage $\mathcal{F}_{s}$ est de
codimension $s.$

Pour $k=1$, le drapeau d'extension $\mathcal{D}_{_{\mathcal{F}_{q}\text{ }%
}}^{k}$ sera dit complet et sera not\'{e} $\mathcal{D}_{_{\mathcal{F}_{q}%
\text{ }}}.$

Si chaque feuilletage $\mathcal{F}_{s}$ est riemannien, le drapeau
d'extension $\mathcal{D}_{_{\mathcal{F}_{q}\text{ }}}^{k}$ sera appel\'{e}
drapeau d'extension riemannienne du feuilletage $\mathcal{F}_{q}.$
\end{definition}

\begin{definition}
Un feuilletage $\mathcal{F}$ est dit transversalement diagonal si et
seulement s'il est d\'{e}fini par un cocycle feuillet\'{e} $%
(U_{_{i}},f_{_{i}},T,\gamma _{_{ij}})_{i\in I}$ tel que les ouverts $%
U_{_{i}} $ soient $\mathcal{F}$-distingu\'{e}s et sur chaque ouvert $%
f_{_{i}}\left( U_{_{i}}\right) $ il existe un syst\`{e}me $\left(
y_{1,}^{i}y_{2,...,}^{i}y_{q}^{i}\right) $ de coordonn\'{e}s $\mathcal{F}$%
-transverses locales tel que relativement aux syst\`{e}mes $\left( \left(
y_{1,}^{i}y_{2,...,}^{i}y_{q}^{i}\right) \right) _{i\in I}$ \ de coordonn%
\'{e}s $\mathcal{F}$-transverses locales sur les ouverts $f_{_{i}}\left(
U_{_{i}}\right) $, la matrice jacobienne $J_{\gamma _{_{ij}}}$ de $\gamma
_{_{ij}}$ soit diagonale.
\end{definition}

\begin{proposition}
Soit $\mathcal{F}_{q}$ un feuilletage transversalement diagonal de
codimension $q$ sur une vari\'{e}t\'{e} $M$.

Alors\ $\mathcal{F}_{q}$ admet un drapeau complet d'extension $\mathcal{D}%
_{_{\mathcal{F}_{q}\text{ }}}=\left( \mathcal{F}_{q-1},\text{ }\mathcal{F}%
_{q-2},\text{ }...,\text{ }\mathcal{F}_{1}\right) .$
\end{proposition}

\begin{proof}
Soit $(U_{_{i}},f_{_{i}}^{q},T^{q},\gamma _{_{ji}}^{q})_{i\in I}$ un cocyle
feuillet\'{e} d\'{e}finissant le feuilletage transversalement diagonal $%
\mathcal{F}_{q}$ de codimension $q$ et $\left( \left(
y_{1,}^{i}y_{2,...,}^{i}y_{q}^{i}\right) \right) _{i\in I}$ les syst\`{e}mes
de coordonn\'{e}es $\mathcal{F}_{q}-$transverses sur les $f_{_{i}}^{q}\left(
U_{_{i}}\right) $ suivant lesquels $J_{\gamma _{_{ij}}^{q}}$ est diagonal.

Notons que $J_{\gamma _{_{ji}}^{q}}=\left( \varepsilon _{jirs}^{q}\right)
_{rs}$ \'{e}tant une matrice inversible \`{a} tous ses \'{e}l\'{e}ments
diagonaux non nuls.

Ceci dit, consid\'{e}rons sur chaque $f_{_{i}}^{q}\left( U_{_{i}}\right) $
le flot d\'{e}fini par le champ de vecteurs $\frac{\partial }{\partial
y_{q}^{i}}.$

Quitte \`{a} r\'{e}duire la "taille" des ouverts $U_{_{i}},$ on peut consid%
\'{e}rer un recouvrement $(U_{_{i}})_{i\in I}\ $de la vari\'{e}t\'{e} $M\ \ $%
tel que dans chaque $f_{_{i}}^{q}\left( U_{_{i}}\right) $ le flot de $\frac{%
\partial }{\partial y_{q}^{i}}$ est un feuilletage simple.

Soit $\theta _{i}^{q}$ la submersion d\'{e}finissant le flot $\mathcal{F}%
_{\theta _{i}^{q}}$ de $\frac{\partial }{\partial y_{q}^{i}}$ sur $%
f_{_{i}}^{q}\left( U_{_{i}}\right) $ et $\mathcal{U}_{_{\theta _{i}^{q}}}$
la vari\'{e}t\'{e} quotient du feuilletage simple $\mathcal{F}_{\theta
_{i}^{q}}.$

La vari\'{e}t\'{e} $T^{q}$ peut \^{e}tre consid\'{e}r\'{e}e comme une union
disjointe des $f_{_{i}}^{q}\left( U_{_{i}}\right) $. Par cons\'{e}quent nous
pouvons affirmer que les submersions $\theta _{i}^{q}$ d\'{e}finissent une
submersion $\theta ^{q}$sur $T^{q}$ dont la restriction \`{a} chaque $%
f_{_{i}}^{q}\left( U_{_{i}}\right) $ est $\theta _{i}^{q}.$

Comme 
\begin{equation*}
\left( \gamma _{_{ji}}^{q}\right) _{\ast }\left( \frac{\partial }{\partial
y_{q}^{i}}\right) =\left( \varepsilon _{ji}^{q}\right) _{ji}\left( \frac{%
\partial }{\partial y_{q}^{i}}\right) =\varepsilon _{ji11}^{q}.\frac{%
\partial }{\partial y_{q}^{j}}\text{ \ et }\varepsilon _{ji11}^{q}\neq 0
\end{equation*}%
alors les $\gamma _{_{ji}}^{q}$ pr\'{e}servent les fibres de la submersion $%
\theta ^{q}.$

Il r\'{e}sulte de cela $\left( \cite{DAD},\cite{DD}\right) $( \textit{%
cf.prop.2.3 et def.2.2}) que le feuilletage $\mathcal{F}_{q}$ de codimension 
$q$ a une extension $\mathcal{F}_{q-1}$ de codimension $q-1$ d\'{e}finie par
le cocycle feuillet\'{e} $(U_{_{i}}$,$f_{_{i}}^{q-1},T^{q-1},\gamma
_{_{ji}}^{q-1})_{i\in I}$ o\`{u} 
\begin{equation*}
\gamma _{_{ji}}^{q-1}:\mathcal{U}_{_{\theta _{i}^{q}}}\rightarrow \mathcal{U}%
_{_{\theta _{j}^{q}}}
\end{equation*}%
est le diff\'{e}omorphisme induit par le diff\'{e}omorphisme local \newline
$\gamma _{_{ji}}^{q}:f_{_{i}}^{q}\left( U_{_{i}}\right) \rightarrow
f_{_{j}}^{q}\left( U_{j}\right) $ sur les vari\'{e}t\'{e}s quotients locales 
$\mathcal{U}_{_{\theta _{i}^{q}}}$ et $\mathcal{U}_{_{\theta _{j}^{q}}}.$ On
a aussi $\left( \cite{DAD},\cite{DD}\right) $ ( \textit{cf def.2.2}): 
\begin{equation*}
T^{q-1}=\underset{i\in I}{\cup }\theta _{i}^{q}\circ f_{_{i}}^{q}\left(
U_{_{i}}\right) =\underset{i\in I}{\cup }\mathcal{U}_{_{\theta _{i}^{q}}},
\end{equation*}%
\begin{equation*}
f_{_{i}}^{q-1}=\theta _{i}^{q}\circ f_{_{i}}^{q}\text{ et \ }\gamma
_{_{ji}}^{q-1}\circ \theta _{i}^{q}=\theta _{j}^{q}\circ \gamma _{_{ji}}^{q}.
\end{equation*}

Montrons maintenant que le feuilletage $\mathcal{F}_{q-1}$ est
transversalement diagonal.

En utilisant l'\'{e}galit\'{e} $\gamma _{_{ji}}^{q-1}\circ \theta
_{i}^{q}=\theta _{j}^{q}\circ \gamma _{_{ji}}^{q}$ on obtient:

i) 
\begin{equation*}
\left( \gamma _{_{ji}}^{q-1}\right) _{\ast }\circ \left( \theta
_{i}^{q}\right) _{\ast }\left( \frac{\partial }{\partial y_{q}^{i}}\right)
=\left( \theta _{j}^{q}\right) _{\ast }\circ \left( \gamma
_{_{ji}}^{q}\right) _{\ast }\left( \frac{\partial }{\partial y_{q}^{i}}%
\right) =J_{\theta _{j}^{q}}\left( \varepsilon _{ji11}^{q}.\frac{\partial }{%
\partial y_{q}^{j}}\right) =0\text{,}
\end{equation*}

ii) pour $k\neq q,$ 
\begin{eqnarray*}
\left( \gamma _{_{ji}}^{q-1}\right) _{\ast }\circ \left( \theta
_{i}^{q}\right) _{\ast }\left( \frac{\partial }{\partial y_{k}^{i}}\right)
&=&\left( \theta _{j}^{q}\right) _{\ast }\circ \left( \gamma
_{_{ji}}^{q}\right) _{\ast }\left( \frac{\partial }{\partial y_{k}^{i}}%
\right) \\
&=&\left( \theta _{j}^{q}\right) _{\ast }\left( \varepsilon _{ij\left(
q-k+1\right) \left( q-k+1\right) }^{q}.\frac{\partial }{\partial y_{k}^{j}}%
\right) \\
&=&\varepsilon _{ij\left( q-k+1\right) \left( q-k+1\right) }^{q}\left(
\theta _{j}^{q}\right) _{\ast }\left( \frac{\partial }{\partial y_{k}^{j}}%
\right) \text{ \ \ \ \ }\left( \ast \right) .
\end{eqnarray*}

On v\'{e}rifie aisement en utilisant le syst\^{e}me differentielle int\'{e}%
grable%
\begin{equation*}
P_{x}^{qi}=\left\{ \overset{q-1}{\underset{k=1}{\sum }}f_{ik}\left( x\right)
.\frac{\partial }{\partial y_{k}^{i}}\left( x\right) \text{ }/\text{ }%
f_{ik}:f_{_{i}}^{q}\left( U_{_{i}}\right) \longrightarrow 
\mathbb{R}
\text{ de classe }C^{\infty }\right\}
\end{equation*}%
sur $f_{i}^{q}(U_{i})$\ que pour $k\neq q$, les $q-1$ champs de vecteurs $%
\left( \theta _{i}^{q}\right) _{\ast }\left( \frac{\partial }{\partial
y_{k}^{i}}\right) $ sur $\mathcal{U}_{_{\theta _{i}^{q}}}$ d\'{e}finissent
un syst\`{e}me de coordonn\'{e}es $\mathcal{F}_{\theta _{i}^{q}}$-transverse
sur $\mathcal{U}_{_{\theta _{i}^{q}}}.$

Soit $\left( z_{q-1}^{i},\text{..., }z_{1}^{i}\right) $ ce syst\`{e}me de
coordonn\'{e}es $\mathcal{F}_{\theta _{i}^{q}}$-transverse sur $\mathcal{U}%
_{_{\theta _{i}^{q}}}.$ On a $\left( \theta _{i}^{q}\right) _{\ast }\left( 
\frac{\partial }{\partial y_{k}^{i}}\right) =\frac{\partial }{\partial
z_{k}^{i}}.$ Et, en utilisant les \'{e}galit\'{e}s \ $\left( \ast \right) $
on obtient 
\begin{equation*}
\left( \gamma _{_{ji}}^{q-1}\right) _{\ast }\left( \frac{\partial }{\partial
z_{k}^{i}}\right) =\varepsilon _{ij\left( q-k+1\right) \left( q-k+1\right)
}^{q}\left( \frac{\partial }{\partial z_{k}^{j}}\right) .
\end{equation*}%
Ainsi le feuilletage $\mathcal{F}_{q-1}$ est transversalement diagonal. En
fait, relativement au syst\`{e}me de coordonn\'{e}es $\left( z_{q-1}^{i},%
\text{..., }z_{1}^{i}\right) $, la matrice diagonale$J_{\gamma
_{_{ji}}^{q-1}}$ est la matrice $J_{\gamma _{_{ji}}^{q}}$ priv\'{e}e de sa
premi\`{e}re ligne et de sa premi\`{e}re colonne.

On construit $\mathcal{F}_{q-2},$ $\mathcal{F}_{q-3\text{ }},$ ... , $%
\mathcal{F}_{1}$ en utilisant la m\^{e}me technique de construction de $%
\mathcal{F}_{q-1}.$
\end{proof}

En utilisant le th\'{e}or\`{e}me 2.4 de d\'{e}composition de
Blumenthal-Hebda $\cite{BH},$ on montre le th\'{e}or\`{e}me suivant en
utilisant les m\^{e}mes techniques de demonstration du th\'{e}or\`{e}me 2.5
dans $\cite{DIA2}$ :

\begin{theorem}
Soit $\mathcal{D}_{_{\mathcal{F}_{q}\text{ }}}=\left( \mathcal{F}_{q-1},%
\text{ }\mathcal{F}_{q-2},\text{ }...,\text{ }\mathcal{F}_{1}\right) $ un
drapeau complet riemannien d'extension d'un feuilletage riemannien $\mathcal{%
F}_{q}$ de codimension $q$ sur une vari\'{e}t\'{e} compacte connexe $M$ et
pour tout $k\in \left\{ 1,...,q\right\} ,$ $\widetilde{\mathcal{F}}_{k}$ d%
\'{e}signe le feuilletage relev\'{e} de $\mathcal{F}_{k}$ sur le rev\^{e}%
tement universel $\widetilde{M}$ de $M.$

S'il existe une m\'{e}trique $g$ quasi-fibr\'{e}e commune pour $\mathcal{F}%
_{q}$ et pour chaque feuilletage du drapeau complet d'extension $\mathcal{D}%
_{_{\mathcal{F}_{q}\text{ }}}$ alors:

i) Chaque feuilletage $\mathcal{F}_{k}$ est transversalement int\'{e}grable
et transversalement parall\'{e}lisable. Les champs de vecteurs du parall\'{e}%
lisme $\mathcal{F}_{k}-$transverse $\left( Y_{s}\right) _{0\leq s\leq k-1%
\text{ \ \ }}$ sont orthogonaux et chaque champ de vecteurs $Y_{s}$ est une
section\ unitaire de $\left( T\mathcal{F}_{s+1}\right) ^{\perp }\cap \left( T%
\mathcal{F}_{s}\right) $ .

ii) Le drapeau $\mathcal{D}_{_{\mathcal{F}_{q}\text{ }}}$ se rel\`{e}ve sur $%
\widetilde{M}$ en un drapeau simple \newline
$\mathcal{D}_{_{\widetilde{\mathcal{F}}_{q}\text{ }}}=\left( \widetilde{%
\mathcal{F}}_{q-1},\widetilde{\mathcal{F}}_{q-2},\text{ }...,\text{ }%
\widetilde{\mathcal{F}}_{1}\right) $ d'extension riemannien du feuilletage
riemannien $\widetilde{\mathcal{F}}_{q}.$
\end{theorem}

On notera que le parall\'{e}lisme $\mathcal{F}_{q}-$transverse $\left(
Y_{s}\right) _{0\leq s\leq q-1\text{ \ \ }}$ sera dit associ\'{e} \`{a} $%
\mathcal{D}_{_{\mathcal{F}_{q}\text{ }}}.$

Le r\'{e}sultat qui suit est la version locale du th\'{e}or\`{e}me de d\'{e}%
composition de Blumenthal-Hebda $\cite{BH}:$

\begin{theorem}
Soit $\mathcal{F}$ un feuilletage riemannien transversalement int\'{e}%
grable\ sur une vari\'{e}t\'{e} riemannienne connexe compacte $\left(
M,g\right) $.

Alors il existe un recouvrement $(U_{_{i}})_{i\in I}$ d'ouverts de la vari%
\'{e}t\'{e} $M$ tel que:

i) Chaque ouvert $U_{_{i}}$ est diff\'{e}omorphe au produit $L_{i}\times
L_{i}^{\bot }$ o\`{u} $L_{i}$ et $L_{i}^{\bot }$ sont respectivement des
feuilles de $\mathcal{F}_{U_{_{i}}}$ et $\mathcal{F}_{U_{_{i}}}^{\bot }$ o%
\`{u} $\mathcal{F}_{U_{_{i}}}$ et $\mathcal{F}_{U_{_{i}}}^{\bot }$ sont
respectivement les restrictions respectives de $\mathcal{F}$ et $\mathcal{F}%
^{\bot }$ \`{a} $U_{_{i}}.$

ii) Toute feuille de $\mathcal{F}_{U_{_{i}}}$ est diff\'{e}omorphe \`{a} $%
L_{i}\times \left\{ p\right\} $ o\`{u} $p\in L_{i}^{\bot }$ et\ toute
feuille de $\mathcal{F}_{U_{_{i}}}^{\bot }$ est diff\'{e}omorphe \`{a} $%
\left\{ p\right\} \times L_{i}^{\bot }$ o\`{u} $p\in L_{i}.$
\end{theorem}

\begin{proof}
i) Soit $p:\widetilde{M}\rightarrow M\ $le rev\^{e}tement universel de $M,$ $%
\pi _{_{1}}\left( M\right) $ le groupe fondamental de $M$,\ $\widetilde{%
\mathcal{F}}$ le relev\'{e} de $\mathcal{F}$ sur $\widetilde{M}$, $%
(U_{_{i}})_{i\in I}$ un recouvrement d'ouverts de la vari\'{e}t\'{e} $M$ par
des ouverts de trivialisation local du rev\^{e}tement universel $\widetilde{M%
}$ de$\ M$, $\mathcal{F}_{U_{_{i}}}$ et $\mathcal{F}_{U_{_{i}}}^{\bot }$ les
restrictions respectives de $\mathcal{F}$ et $\mathcal{F}^{\bot }$ \`{a} $%
U_{_{i}}.$

On a $p^{-1}\left( U_{_{i}}\right) =\underset{k\in K}{\cup }\widetilde{U}%
_{_{i}}^{k}.$

Dans tout ce qui suit on fixe $a\in K.$

Quitte \`{a} r\'{e}duire la "taille" de l'ouvert $U_{_{i}},$ on peut
supposer en utilisant le th\'{e}or\`{e}me 2.4 de d\'{e}composition de
Blumenthal-Hebda qu'il existe un diff\'{e}omorphisme $\phi _{iaa}:\widetilde{%
U}_{_{i}}^{a}\rightarrow \widetilde{L}_{ia}\times \widetilde{L}_{ia}^{\bot }$
o\`{u} $\widetilde{L}_{ia}$ est une feuille de la restriction $\widetilde{%
\mathcal{F}}_{\widetilde{U}_{_{i}}^{a}}$ de $\widetilde{\mathcal{F}}$ \`{a} $%
\widetilde{U}_{_{i}}^{a}$ et $\widetilde{L}_{ia}^{\bot }$ est une feuille de
la restriction $\widetilde{\mathcal{F}}_{\widetilde{U}_{_{i}}^{a}}^{\bot }$
de $\widetilde{\mathcal{F}}^{\bot }$ \`{a} $\widetilde{U}_{_{i}}^{a}.$

Soit $b\in K.$ Comme $p\left( \widetilde{U}_{_{i}}^{a}\right) =p\left( 
\widetilde{U}_{_{i}}^{b}\right) =U_{_{i}}$ alors il existe \newline
$\sigma _{iab}\in \pi _{_{1}}\left( M\right) $ tel que 
\begin{equation*}
\sigma _{iab}\left( \widetilde{U}_{_{i}}^{a}\right) =\widetilde{U}%
_{_{i}}^{b}.
\end{equation*}%
Par construction des feuilletages relev\'{e}s $\widetilde{\mathcal{F}}$ et $%
\widetilde{\mathcal{F}}^{\bot }$ on a 
\begin{equation*}
\sigma _{iab}\left( \widetilde{\mathcal{F}}_{\widetilde{U}%
_{_{i}}^{a}}\right) =\widetilde{\mathcal{F}}_{\widetilde{U}_{_{i}}^{b}}\text{
et\ }\sigma _{iab}\left( \widetilde{\mathcal{F}}_{\widetilde{U}%
_{_{i}}^{a}}^{\bot }\right) =\widetilde{\mathcal{F}}_{\widetilde{U}%
_{_{i}}^{b}}^{\bot }.
\end{equation*}

Ainsi $\sigma _{iab}\left( \widetilde{L}_{ia}\right) $ est une feuille de $%
\widetilde{\mathcal{F}}_{\widetilde{U}_{_{i}}^{b}}$ et $\sigma _{iab}\left( 
\widetilde{L}_{ia}^{\bot }\right) $ est une feuille de $\widetilde{\mathcal{F%
}}_{\widetilde{U}_{_{i}}^{b}}^{\bot }.$

Consid\'{e}rons $\sigma _{iab}^{1}$ la restriction de $\sigma _{iab}$ \`{a} $%
\widetilde{L}_{ia}$ et $\sigma _{iab}^{2}$ la restriction de $\sigma _{iab}$ 
\`{a} $\widetilde{L}_{ia}^{\bot }.$

On a $\sigma _{iab}^{1}:\widetilde{L}_{ia}\rightarrow \sigma _{iab}\left( 
\widetilde{L}_{ia}\right) $ et $\sigma _{iab}^{2}:\widetilde{L}_{ia}^{\bot
}\rightarrow \sigma _{iab}\left( \widetilde{L}_{ia}^{\bot }\right) $ qui
sont des diff\'{e}omorphismes. Par cons\'{e}quent $\sigma _{iab}$ induit un
diff\'{e}omorphisme 
\begin{equation*}
\overline{\sigma }_{iab}:\widetilde{L}_{ia}\times \widetilde{L}_{ia}^{\bot
}\rightarrow \sigma _{iab}^{1}\left( \widetilde{L}_{ia}\right) \times \sigma
_{iab}^{2}\left( \widetilde{L}_{ia}^{\bot }\right)
\end{equation*}%
tel que 
\begin{equation*}
\overline{\sigma }_{iab}=\left( \sigma _{iab}^{1}\circ p_{ia}^{1},\sigma
_{iab}^{2}\circ p_{ia}^{2}\right)
\end{equation*}%
o\`{u} $p_{ia}^{1}:\widetilde{L}_{ia}\times \widetilde{L}_{ia}^{\bot
}\rightarrow \widetilde{L}_{ia}$ est la premi\`{e}re projection et $%
p_{ia}^{2}:\widetilde{L}_{ia}\times \widetilde{L}_{ia}^{\bot }\rightarrow 
\widetilde{L}_{ia}$ est la seconde projection.

Il r\'{e}sulte de ce qui pr\'{e}c\`{e}de qu'il existe un diff\'{e}omorphisme 
$\phi _{iab}$ qui rend le diagramme suivant commutatif:%
\begin{equation*}
\begin{array}{ccc}
\widetilde{U}_{_{i}}^{a} & \overset{\phi _{iaa}}{\rightarrow } & \widetilde{L%
}_{ia}\times \widetilde{L}_{ia}^{\bot } \\ 
\downarrow \sigma _{iab} &  & \downarrow \overline{\sigma }_{iab} \\ 
\widetilde{U}_{_{i}}^{b} & \overset{\phi _{iab}}{\rightarrow } & \sigma
_{iab}^{1}\left( \widetilde{L}_{ia}\right) \times \sigma _{iab}^{2}\left( 
\widetilde{L}_{ia}^{\bot }\right)%
\end{array}%
.
\end{equation*}

On construit ainsi un diff\'{e}omorphisme de $p^{-1}\left( U_{_{i}}\right) =%
\underset{k\in K}{\cup }\widetilde{U}_{_{i}}^{k}$ sur\newline
$\underset{k\in K}{\cup }\phi _{ik}\left( \widetilde{U}_{_{i}}^{k}\right) =%
\underset{k\in K}{\cup }\sigma _{iak}^{1}\left( \widetilde{L}_{ia}\right)
\times \sigma _{iak}^{2}\left( \widetilde{L}_{ia}^{\bot }\right) $ dont la
restriction \`{a} $\widetilde{U}_{_{i}}^{b}$ est $\phi _{iab}.$ Ce diff\'{e}%
omorphisme ne d\'{e}pend pas du choix de l'ouvert $\widetilde{U}_{_{i}}^{a}$
mais seulement de $U_{_{i}}$, de la feuille 
\begin{equation*}
L_{i}=p\left( \sigma _{iab}^{1}\left( \widetilde{L}_{ia}\right) \right)
=p\left( \widetilde{L}_{ia}\right)
\end{equation*}
de $\mathcal{F}_{U_{_{i}}}$ et de la feuille 
\begin{equation*}
L_{i}^{\bot }=p\left( \sigma _{iab}^{2}\left( \widetilde{L}_{ia}^{\bot
}\right) \right) =p\left( \widetilde{L}_{ia}^{\bot }\right)
\end{equation*}
de $\mathcal{F}_{U_{_{i}}}^{\bot }.$ Dans la suite, ce diff\'{e}omorphisme\
ce notera $\phi _{i}.$

On peut donc dire en remarquant que $\sigma _{iaa}=Id,$ $\sigma
_{iaa}^{1}=Id $ et $\sigma _{iaa}^{2}=Id$ qu'on a $\sigma _{iab}\left( 
\widetilde{U}_{_{i}}^{a}\right) =\widetilde{U}_{_{i}}^{b}$, 
\begin{equation*}
\phi _{i}\left( \widetilde{U}_{_{i}}^{b}\right) =\sigma _{iab}^{1}\left( 
\widetilde{L}_{ia}\right) \times \sigma _{iab}^{2}\left( \widetilde{L}%
_{ia}^{\bot }\right) \text{ et }\overline{\sigma }_{iab}\circ \phi _{i}=\phi
_{i}\circ \sigma _{iab}\text{\ \ }
\end{equation*}%
et l'\'{e}galit\'{e}%
\begin{equation*}
\overline{\sigma }_{iab}=\left( \sigma _{iab}^{1}\circ p_{ia}^{1},\sigma
_{iab}^{2}\circ p_{ia}^{2}\right)
\end{equation*}%
entraine que%
\begin{eqnarray*}
\frac{\phi _{i}\left( p^{-1}\left( U_{_{i}}\right) \right) }{\pi
_{_{1}}\left( M\right) } &=&\frac{\underset{k\in K}{\cup }\sigma
_{iak}^{1}\left( \widetilde{L}_{ia}\right) \times \sigma _{iak}^{2}\left( 
\widetilde{L}_{ia}^{\bot }\right) }{\pi _{_{1}}\left( M\right) } \\
&=&\frac{\underset{k\in K}{\cup }\sigma _{iak}^{1}\left( \widetilde{L}%
_{ia}\right) }{\pi _{_{1}}\left( M\right) }\times \frac{\underset{k\in K}{%
\cup }\sigma _{iak}^{2}\left( \widetilde{L}_{ia}^{\bot }\right) }{\pi
_{_{1}}\left( M\right) }.
\end{eqnarray*}

Par construction de $\widetilde{\mathcal{F}}$ et $\widetilde{\mathcal{F}}%
^{\bot }$, $\sigma _{iak}^{1}$ \'{e}tant la restriction de $\sigma _{iak}$ 
\`{a} $\widetilde{L}_{ia}$ et $\sigma _{iak}^{2}$ la restriction de $\sigma
_{iak}$ \`{a} $\widetilde{L}_{ia}^{\bot },$ on a: 
\begin{equation*}
\frac{\underset{k\in I}{\cup }\sigma _{iak}^{1}\left( \widetilde{L}%
_{ia}\right) }{\pi _{_{1}}\left( M\right) }=L_{i}\text{ et }\frac{\underset{%
k\in I}{\cup }\sigma _{iak}^{2}\left( \widetilde{L}_{ia}^{\bot }\right) }{%
\pi _{_{1}}\left( M\right) }=L_{i}^{\bot }.
\end{equation*}

L'\'{e}galit\'{e} $\overline{\sigma }_{iab}\circ \phi _{i}=\phi _{i}\circ
\sigma _{iab}$ montre que $\phi _{i}$ est invariant par l'action de $\pi
_{_{1}}\left( M\right) .$ Par cons\'{e}quent $\phi _{i}$ passe au quotient
sous l'action de $\pi _{_{1}}\left( M\right) $ en un diff\'{e}omorphisme 
\begin{equation*}
\overline{\phi }_{i}:\frac{p^{-1}\left( U_{_{i}}\right) }{\pi _{_{1}}\left(
M\right) }\rightarrow \frac{\phi _{i}\left( p^{-1}\left( U_{_{i}}\right)
\right) }{\pi _{_{1}}\left( M\right) }.
\end{equation*}

Comme $\frac{p^{-1}\left( U_{_{i}}\right) }{\pi _{_{1}}\left( M\right) }%
=U_{_{i}}$ et 
\begin{eqnarray*}
\frac{\phi _{i}\left( p^{-1}\left( U_{_{i}}\right) \right) }{\pi
_{_{1}}\left( M\right) } &=&\frac{\underset{k\in K}{\cup }\sigma
_{iak}^{1}\left( \widetilde{L}_{ia}\right) }{\pi _{_{1}}\left( M\right) }%
\times \frac{\underset{k\in K}{\cup }\sigma _{iak}^{2}\left( \widetilde{L}%
_{ia}^{\bot }\right) }{\pi _{_{1}}\left( M\right) } \\
&=&L_{i}\times L_{i}^{\bot }
\end{eqnarray*}%
alors $\phi _{i}$ passe au quotient sous l'action de $\pi _{_{1}}\left(
M\right) $ en un diff\'{e}omorphisme 
\begin{equation*}
\overline{\phi }_{i}:U_{_{i}}\rightarrow L_{i}\times L_{i}^{\bot }.
\end{equation*}

ii) Soit $\widetilde{\mathcal{F}}_{p^{-1}\left( U_{_{i}}\right) }$ et $%
\widetilde{\mathcal{F}}_{p^{-1}\left( U_{_{i}}\right) }^{\bot }$ les
restrictions respectives de $\widetilde{\mathcal{F}}$ et $\widetilde{%
\mathcal{F}}^{\bot }$ \`{a} $p^{-1}\left( U_{_{i}}\right) .$

Pour $a\in K$ fix\'{e}, on a pour tout $k\in K$

\begin{equation*}
\sigma _{iak}\left( \widetilde{\mathcal{F}}_{\widetilde{U}%
_{_{i}}^{a}}\right) =\widetilde{\mathcal{F}}_{\widetilde{U}_{_{i}}^{k}}\text{
et }\sigma _{iak}\left( \widetilde{\mathcal{F}}_{\widetilde{U}%
_{_{i}}^{a}}^{\bot }\right) =\widetilde{\mathcal{F}}_{\widetilde{U}%
_{_{i}}^{k}}^{\bot }.
\end{equation*}%
D'o\`{u} l'\'{e}galit\'{e} $\overline{\sigma }_{iak}\circ \phi _{i}=\phi
_{i}\circ \sigma _{iak}$ montre que 
\begin{equation*}
\overline{\sigma }_{iak}\left( \phi _{i}\left( \widetilde{\mathcal{F}}_{%
\widetilde{U}_{_{i}}^{a}}\right) \right) =\phi _{i}\left( \widetilde{%
\mathcal{F}}_{\widetilde{U}_{_{i}}^{k}}\right) \text{ et }\overline{\sigma }%
_{iak}\left( \phi _{i}\left( \widetilde{\mathcal{F}}_{\widetilde{U}%
_{_{i}}^{a}}^{\bot }\right) \right) =\phi _{i}\left( \widetilde{\mathcal{F}}%
_{\widetilde{U}_{_{i}}^{k}}^{\bot }\right) .
\end{equation*}%
Par cons\'{e}quent les feuilletages $\phi _{i}\left( \widetilde{\mathcal{F}}%
_{p^{-1}\left( U_{_{i}}\right) }\right) $ et $\phi _{i}\left( \widetilde{%
\mathcal{F}}_{p^{-1}\left( U_{_{i}}\right) }^{\bot }\right) $ sont invariant
par l'action de $\pi _{_{1}}\left( M\right) $ et passe donc au quotient en
un feuilletage sur $L_{i}\times L_{i}^{\bot }.$

Comme d'apr\`{e}s le th\'{e}or\`{e}me 2.4 de d\'{e}composition de
Blumenthal-Hebda les feuilles du feuilletage simple $\phi _{i}\left( 
\widetilde{\mathcal{F}}_{\widetilde{U}_{_{i}}^{k}}\right) $ sont $\sigma
_{iak}^{1}\left( \widetilde{L}_{ia}\right) \times \left\{ p\right\} $ o\`{u} 
$\widetilde{p}\in \sigma _{iak}^{2}\left( \widetilde{L}_{ia}^{\bot }\right) $
et celles de $\phi _{i}\left( \widetilde{\mathcal{F}}_{\widetilde{U}%
_{_{i}}^{k}}^{\bot }\right) $ sont $\sigma _{iak}^{2}\left( \widetilde{L}%
_{ia}^{\bot }\right) \times \left\{ p\right\} $ o\`{u} $p\in \sigma
_{iak}^{1}\left( \widetilde{L}_{ia}\right) $ et comme 
\begin{equation*}
\overline{\sigma }_{iak}=\left( \sigma _{iak}^{1}\circ p_{ia}^{1},\sigma
_{iak}^{2}\circ p_{ia}^{2}\right) ,\text{ }
\end{equation*}
\begin{equation*}
\text{ }\frac{\underset{k\in K}{\cup }\sigma _{iak}^{1}\left( \widetilde{L}%
_{ia}\right) }{\pi _{_{1}}\left( M\right) }=L_{i}\text{\ et }\frac{\underset{%
k\in K}{\cup }\sigma _{iak}^{2}\left( \widetilde{L}_{ia}^{\bot }\right) }{%
\pi _{_{1}}\left( M\right) }=L_{i}^{\bot },
\end{equation*}%
alors le feuilletage quotient $\frac{\phi _{i}\left( \widetilde{\mathcal{F}}%
_{p^{-1}\left( U_{_{i}}\right) }\right) }{\pi _{_{1}}\left( M\right) }$ \`{a}
ses feuilles qui sont $L_{i}\times \left\{ x\right\} $ o\`{u} $x\in
L_{i}^{\bot }$ et $\frac{\phi _{i}\left( \widetilde{\mathcal{F}}_{\widetilde{%
U}_{_{i}}^{k}}^{\bot }\right) }{\pi _{_{1}}\left( M\right) }$ \`{a} ses
feuilles qui sont $L_{i}^{\bot }\times \left\{ x\right\} $ o\`{u} $x\in
L_{i}.$
\end{proof}

\begin{definition}
On dit qu'un feuilletage $\mathcal{F}$ est riemannien transversalement
diagonal relativement \`{a} une m\'{e}trique $g_{_{T}}$ si et seulement s'il
est d\'{e}fini par un cocycle feuillet\'{e} $(U_{_{i}},f_{_{i}},T,\gamma
_{_{ij}})_{i\in I}$ v\'{e}rifiant les conditions suivantes:

i) les ouverts $U_{_{i}}$ sont $\mathcal{F}$-distingu\'{e}s,

ii) la m\'{e}trique $g_{_{T}}$ est une m\'{e}trique sur la vari\'{e}t\'{e}
transverse $T$ et les $\gamma _{_{ij}}$ sont des isom\'{e}tries locales pour
cette m\'{e}trique,

\textit{iii)} sur chaque ouvert $f_{_{i}}\left( U_{_{i}}\right) $ il existe
un syst\`{e}me $\left( y_{1,}^{i}y_{2,...,}^{i}y_{q}^{i}\right) $ de coordonn%
\'{e}s $\mathcal{F}$-transverses locales tel que relativement aux syst\`{e}%
mes de coordonn\'{e}s $\mathcal{F}$-transverses locales $\left( \left(
y_{1,}^{i}y_{2,...,}^{i}y_{q}^{i}\right) \right) _{i\in I}$ sur les ouverts $%
f_{_{i}}\left( U_{_{i}}\right) $, la matrice jacobienne $J_{\gamma _{_{ij}}}$
de $\gamma _{_{ij}}$ est diagonale et orthogonale.

Les syst\`{e}mes $\left( \left( y_{1,}^{i}y_{2,...,}^{i}y_{q}^{i}\right)
\right) _{i\in I}$ de coordonn\'{e}s $\mathcal{F}$-transverses locales sur
les ouverts $f_{_{i}}\left( U_{_{i}}\right) $ seront appel\'{e}s syst\`{e}%
mes de coordonn\'{e}s transverses diagonaux de $\mathcal{F}$.
\end{definition}

On notera que ce qui suit:

\textit{\ }Si $(M_{1},\mathcal{F}_{1}),$ $(M_{2},\mathcal{F}_{2}),...,$ $%
(M_{k},\mathcal{F}_{k})$ sont $k$ feuilletages riemanniens de codimension un
alors le feuilletage produit $\left( M_{1}\times M_{2}\times ...\times M_{k},%
\text{ }\mathcal{F}_{1}\times \mathcal{F}_{2}\times ...\times \mathcal{F}%
_{k}\right) $ est un feuilletage riemannien transversalement diagonal de
codimension $k$ sur la vari\'{e}t\'{e} produit $M=M_{1}\times M_{2}\times
...\times M_{k}.$

Dans la preuve du r\'{e}sultat qui suit on donne une caract\'{e}risation
matricielle des drapeaux complets riemanniens d'extension d'un feuilletage
riemannien.

\begin{theorem}
Soit $\mathcal{D}_{_{\mathcal{F}_{q}\text{ }}}=\left( \mathcal{F}_{q-1},%
\text{ }\mathcal{F}_{q-2},\text{ }...,\text{ }\mathcal{F}_{1}\right) $ un
drapeau complet riemannien d'extension d'un feuilletage riemannien $\mathcal{%
F}_{q}$ de codimension $q$ sur une vari\'{e}t\'{e} compacte connexe $M$ tel
qu'il existe sur la vari\'{e}t\'{e} $M$ une m\'{e}trique $g$ quasi-fibr\'{e}
commune \`{a} tous les feuilletages riemanniens $\mathcal{F}_{k}$.

Alors chaque feuilletage riemanien $\mathcal{F}_{k}$ est transversalement
diagonal.
\end{theorem}

\begin{proof}
Soit $p:\widetilde{M}\rightarrow M\ $le rev\^{e}tement universel $M,$\ $%
\widetilde{\mathcal{F}}_{k}$ le relev\'{e} de $\mathcal{F}_{k}$ sur $%
\widetilde{M}$.

Consid\'{e}rons $(U_{_{i}})_{i\in I}$ un recouvrement d'ouverts de la vari%
\'{e}t\'{e} $M$ par des ouverts de trivialisation local du rev\^{e}tement
universel $\widetilde{M}$ de$\ M$ qui sont distingu\'{e}s pour chaque
feuilletage $\mathcal{F}_{k}$ et soit $\mathcal{F}_{k/U_{_{i}}}$ et $%
\mathcal{F}_{k/U_{_{i}}}^{\bot }$ les restrictions respectives de $\mathcal{F%
}_{k}$ et $\mathcal{F}_{k}^{\bot }$ \`{a} $U_{_{i}}.$

D'apr\`{e}s le th\'{e}or\`{e}me 4.3 chaque ouvert $U_{_{i}}$ est diff\'{e}%
omorphe ($\simeq $) au produit $L_{q}^{i}\times L_{q}^{i\bot }$ o\`{u} $%
L_{q}^{i}$ est une feuille de $\mathcal{F}_{q/U_{_{i}}}$ et $L_{q}^{i\bot }$
est une feuille de $\mathcal{F}_{q/U_{_{i}}}^{\bot }.$

Soit $\widetilde{L}_{q}^{\bot }$ une feuille de $\widetilde{\mathcal{F}}_{q}$
au "dessus" de $U_{_{i}}$ qui se projette sur $U_{_{i}}$ suivant $%
L_{q}^{i\bot }.$

On montre dans $\cite{DIA2}$ que le drapeau relev\'{e} $\mathcal{D}_{_{%
\widetilde{\mathcal{F}}_{q}\text{ }}}=\left( \widetilde{\mathcal{F}}_{q-1},%
\widetilde{\mathcal{F}}_{q-2},\text{ }...,\text{ }\widetilde{\mathcal{F}}%
_{1}\right) $ de $\mathcal{D}_{_{\mathcal{F}_{q}\text{ }}}$ sur $\widetilde{M%
}$\ se proj\`{e}te sur $\widetilde{L}_{q}^{\bot }$ en un drapeau complet
simple 
\begin{equation*}
\mathcal{D}_{_{\widetilde{\mathcal{F}}_{q-1}\text{ }}}^{\widetilde{L}%
_{q}^{\bot }}=\left( \widetilde{\mathcal{F}}_{q-2}^{\widetilde{L}_{q}^{\bot
}},\widetilde{\mathcal{F}}_{q-3}^{\widetilde{L}_{q}^{\bot }},\text{ }...,%
\text{ }\widetilde{\mathcal{F}}_{1}^{\widetilde{L}_{q}^{\bot }}\right)
\end{equation*}
d'extension riemannienne du feuilletage $\widetilde{\mathcal{F}}_{q-1}^{%
\widetilde{L}_{q}^{\bot }}$ o\`{u} $\widetilde{\mathcal{F}}_{k}^{\widetilde{L%
}_{q}^{\bot }}$ d\'{e}signe le feuilletage projet\'{e} sur $\widetilde{L}%
_{q}^{\bot }$ du feuilletage relev\'{e} $\widetilde{\mathcal{F}}_{k}$ de $%
\mathcal{F}_{k}$ sur $\widetilde{M}.$ On v\'{e}rifie aisement que $%
\widetilde{\mathcal{F}}_{q-1}^{\widetilde{L}_{q}^{\bot }}$ est un flot
riemannien.

Ainsi comme $\widetilde{L}_{q}^{\bot }$ est une variet\'{e} compl\`{e}te
connexe et simplement connexe le th\'{e}or\`{e}me 2.4 de d\'{e}composition
de Blumenthal-Hebda $\cite{BH}$ assure qu'il existe un diff\'{e}omorphe $%
\phi ^{2}:\widetilde{L}_{q}^{\bot }$ $\rightarrow $ $\widetilde{H}%
_{Y_{q-1}}\times \widetilde{H}_{Y_{q-1}}^{\bot }$ o\`{u} $\widetilde{H}%
_{Y_{q-1}}$ est une feuille du flot riemannien $\widetilde{\mathcal{F}}%
_{q-1}^{\widetilde{L}_{q}^{\bot }}$ et $\widetilde{H}_{Y_{q-1}}^{\bot }$ est
une feuille du feuilletage $\left( \widetilde{\mathcal{F}}_{q-1}^{\widetilde{%
L}_{q}^{\bot }}\right) ^{\bot }.$

En appliquant le th\'{e}or\`{e}me 4.3 on obtient qu'il existe un diff\'{e}%
omorphisme $\phi _{i}^{2}:p^{-1}\left( L_{q}^{i\bot }\right) \rightarrow
\phi ^{2}\left( p^{-1}\left( L_{q}^{i\bot }\right) \right) $ qui passe au
quotient par l'action de $\pi _{_{1}}\left( \widetilde{L}_{q}^{\bot }\right)
.$ (\textit{cf.preuve th\'{e}o.4.3}) en un diff\'{e}omorphisme $\overline{%
\phi }_{i}^{2}:L_{q}^{i\bot }\rightarrow H_{Y_{q-1}}^{^{i}}\times
H_{Y_{q-1}}^{^{i\bot }}$ o\`{u} $H_{Y_{q-1}}^{^{i}}$ est une feuille du flot 
$\mathcal{F}_{Y_{q-1}}=p\left( \widetilde{\mathcal{F}}_{q-1}^{\widetilde{L}%
_{q}^{\bot }}\right) $ et $H_{Y_{q-1}}^{^{i\bot }}$ est une feuille de $%
p\left( \widetilde{\mathcal{F}}_{q-1}^{\widetilde{L}_{q}^{\bot }}\right)
^{\bot }.$

On obtient ainsi que $U_{_{i}}$ est diff\'{e}omorphe au produit $%
L_{q}^{i}\times H_{Y_{q-1}}^{^{i}}\times H_{Y_{q-1}}^{^{i\bot }}.$

De proche en proche on obtient que chaque ouvert $U_{_{i}}$ est diff\'{e}%
omorphe au produit$\ L_{q}^{^{i}}\times H_{Y_{q-1}}^{^{i}}\times ...\times
H_{Y_{0}}^{^{i}}$ o\`{u} $L_{q}^{^{i}}\ $est une feuille de $\mathcal{F}%
_{q/U_{_{i}}}$, $H_{Y_{k}}^{^{i}}\ $est une feuille de la restriction \ $%
\mathcal{F}_{Y_{k}/U_{_{i}}}$ du flot $\mathcal{F}_{Y_{k}}$ \`{a} $U_{_{i}}.$

On v\'{e}rifie aisement en appliquant le th\'{e}or\`{e}me 4.3 que si on d%
\'{e}signe par $\mathcal{F}_{k/U_{_{i}}}$ la restriction \`{a} $U_{_{i}}$ de 
$\mathcal{F}_{k}$ alors chaque feuille de $\mathcal{F}_{k/U_{_{i}}}$ est de
la forme $L_{k}^{^{i}}\times \left\{ p\right\} $ avec $%
L_{k}^{^{i}}=L_{q}^{^{i}}\times H_{Y_{q-1}}^{^{i}}\times
H_{Y_{q-2}}^{^{i}}\times ...\times H_{Y_{k}}^{^{i}}$ et $\ p\in
H_{k}^{^{i}}=H_{Y_{k-1}}^{^{i}}\times H_{Y_{k-2}}^{^{i}}\times ...\times
H_{Y_{0}}^{^{i}}.$

On v\'{e}rifie aussi aisement que sur chaque feuille $H_{Y_{k}}^{^{i}}$ le
champ de vecteurs unitaire $Y_{k}$ induit un champ de vecteurs $Y_{k}^{i}$
et les champs de vecteurs $Y_{k}^{i}$ sont orthogonaux deux \`{a} deux pour
la m\'{e}trique $g.$

Consid\'{e}rons maintenant $\left( a_{q},a_{q-1},...,a_{0}\right) \in
L_{q}^{^{i}}\times H_{Y_{q-1}}^{^{i}}\times ...\times H_{Y_{0}}^{^{i}}.$

On pose 
\begin{equation*}
\begin{array}{cccc}
R_{h_{k}}: & T_{a_{k}}H_{Y_{k}}^{^{i}} & \rightarrow & T_{a_{q}}L_{q}^{^{i}}%
\times T_{a_{q-1}}H_{Y_{q-1}}^{^{i}}\times ...\times
T_{a_{0}}H_{Y_{0}}^{^{i}} \\ 
& X_{a_{k}}^{i} & \mapsto & \left( X_{a_{k}}^{i}\right) ^{h_{k}}=\left(
0,0,...,0,X_{a_{k}}^{i},0,...,0\right)%
\end{array}%
\end{equation*}%
o\`{u} $X_{a_{k}}^{i}$ est \`{a} la $\left( q-k+1\right) ^{i\grave{e}me}$
position. On remarquera que $q-k+1$ est la position de $TH_{Y_{k}}^{^{i}}$
dans le produit cartesien $TL_{q}^{^{i}}\times TH_{Y_{q-1}}^{^{i}}\times
...\times TH_{Y_{0}}^{^{i}}.$

Soit $\mathcal{X}\left( H_{Y_{k}}^{^{i}}\right) $ l'alg\`{e}bre de lie des
champs de vecteurs tangents \`{a} $H_{Y_{k}}^{^{i}}$ et $X_{k}^{i}\in 
\mathcal{X}\left( H_{Y_{k}}^{^{i}}\right) .$

Les sections $\left( \left( X_{k}^{i}\right) _{a_{k}}\right) ^{h_{k}}$ de $%
T\left( L_{q}^{^{i}}\times H_{Y_{q-1}}^{^{i}}\times ...\times
H_{Y_{0}}^{^{i}}\right) $\ o\`{u}$\ a_{k}\in H_{Y_{k}}^{^{i}}$ d\'{e}%
finissent un champ de vecteurs sur $L_{q}^{^{i}}\times
H_{Y_{q-1}}^{^{i}}\times ...\times H_{Y_{0}}^{^{i}}$ qu'on notera $\left(
X_{k}^{i}\right) ^{h_{k}}.$

On montre dans $\cite{RM}$ que pour $X_{k}^{i}\in TH_{Y_{k}}^{^{i}}$ et $%
X_{s}^{i}\in TH_{Y_{s}}^{^{i}}$ on a 
\begin{equation*}
\left[ \left( X_{k}^{i}\right) ^{h_{q-k+1}},\left( X_{s}^{i}\right)
^{h_{q-s+1}}\right] =0\text{ pour }k\neq s.
\end{equation*}%
Il r\'{e}sulte de cela, $Y_{k}^{i}$ \'{e}tant le champ induit par $Y_{k}$
sur la feuille $H_{Y_{k}}^{^{i}},$ que 
\begin{equation*}
\left[ \left( Y_{k}^{i}\right) ^{h_{q-k+1}},\left( Y_{s}^{i}\right)
^{h_{q-s+1}}\right] =0\text{ pour }k\neq s.
\end{equation*}

Ainsi, l'ouvert $U_{_{i}}$ \'{e}tant un produit des vari\'{e}t\'{e}s $%
L_{q}^{^{i}}$, $H_{Y_{q-1}}^{^{i}}$, $H_{Y_{q-2}}^{^{i}},...,$ et $%
H_{Y_{0}}^{^{i}},$ on a les champs de vecteurs $\left( Y_{s}^{i}\right)
^{h_{q-s+1}}$ sur $U_{_{i}}$ qui d\'{e}finissent un syst\`{e}me de coordonn%
\'{e}s $\mathcal{F}_{q}-$transverse $\left( y_{q-1}^{i},...,y_{0}^{i}\right) 
$ sur $U_{_{i}}.$

Dans la suite $\left( Y_{s}^{i}\right) ^{h_{q-s+1}}$ se notera $\frac{%
\partial }{\partial y_{s}^{i}}.$

Le champ de rep\`{e}res orthonorm\'{e}s $\mathcal{F}_{q}$-transverse local $%
\left( \frac{\partial }{\partial y_{q-1}^{i}},\frac{\partial }{\partial
y_{q-2}^{i}},...,\frac{\partial }{\partial y_{0}^{i}}\right) $ sur $U_{_{i}}$
sera dit compatible avec le drapeau riemannien $\mathcal{D}_{_{\mathcal{F}%
_{q}\text{ }}}.$ On dira aussi que le syst\`{e}me de coordonn\'{e}s $%
\mathcal{F}_{q}-$transverse $\left( y_{q-1}^{i},...,y_{0}^{i}\right) $ est 
\'{e}galement compatible avec $\mathcal{D}_{_{\mathcal{F}_{q}\text{ }}}.$

Soit $(U_{_{i}},f_{_{i}}^{k},T^{k},\gamma _{_{ij}}^{k})_{i\in I}$ un cocyle
feuillet\'{e} d\'{e}finissant le feuilletage riemannien $\mathcal{F}_{k}$\
sur l'ouvert distingu\'{e}s $U_{_{i}}$.

On d\'{e}signe par $\theta _{i}^{k-1}$ la liaison riemannienne entre les
feuilletages riemanniens $\left( U_{_{i}},\mathcal{F}_{k}\right) $ et $%
\left( U_{_{i}},\mathcal{F}_{k-1}\right) .$

Comme chaque feuille de $\mathcal{F}_{k/U_{_{i}}}$ est de la forme $%
L_{k}^{^{i}}\times \left\{ p\right\} $ avec 
\begin{equation*}
L_{k}^{^{i}}=L_{q}^{^{i}}\times H_{Y_{q-1}}^{^{i}}\times
H_{Y_{q-2}}^{^{i}}\times ...\times H_{Y_{k}}^{^{i}}
\end{equation*}%
et $\ p\in H_{k}^{^{i}}=H_{Y_{k-1}}^{^{i}}\times H_{Y_{k-2}}^{^{i}}\times
...\times H_{Y_{0}}^{^{i}}$ alors $\theta _{i}^{k-1}$ est la projection de $%
H_{k}^{^{i}}$ sur $H_{k-1}^{^{i}}$ suivant $H_{Y_{k-1}}^{^{i}}.$

Il en r\'{e}sulte que 
\begin{equation*}
\left( \theta _{i}^{k-1}\right) _{\ast }\left( \frac{\partial }{\partial
y_{k-1}^{i}}\right) =0.
\end{equation*}

On signale avant de continuer que $H_{Y_{k-1}}^{^{i}}\times
H_{Y_{k-2}}^{^{i}}\times ...\times H_{Y_{0}}^{^{i}}$ est diff\'{e}omorphe 
\`{a} une vari\'{e}t\'{e} int\'{e}grable de $\left( T\mathcal{F}%
_{k/U_{_{i}}}\right) ^{\perp }.$ On peut de ce fait identifier $%
f_{i}^{k}(U_{i})\ $et\ $H_{Y_{k-1}}^{^{j}}\times H_{Y_{k-2}}^{^{j}}\times
...\times H_{Y_{0}}^{^{j}}$.

Sous l'identification $f_{i}^{k}(U_{i})\simeq H_{Y_{k-1}}^{^{j}}\times
H_{Y_{k-2}}^{^{j}}\times ...\times H_{Y_{0}}^{^{j}}$on peut consid\'{e}rer
le syst\`{e}me de coordonn\'{e}s $\mathcal{F}_{k}-$transverse $\left(
y_{k-1}^{i},...,y_{0}^{i}\right) $ sur $U_{_{i}}$ comme \'{e}tant un syst%
\`{e}me de coordonn\'{e}s sur $f_{i}^{k}(U_{i})$ qu'on notera encore $\left(
y_{k-1}^{i},...,y_{0}^{i}\right) .$

L'\'{e}galit\'{e} $\left( \theta _{i}^{k-1}\right) _{\ast }\left( \frac{%
\partial }{\partial y_{k-1}^{i}}\right) =0$ montre que 
\begin{equation*}
\theta _{i}^{k-1}\left( y_{k-1}^{i},y_{k-2}^{i},...,y_{0}^{i}\right) =\left(
y_{k-2}^{i},...,y_{0}^{i}\right) .
\end{equation*}

Pour $U_{i}\cap U_{j}\neq \emptyset ,$ soit $\left(
y_{k-1}^{j},y_{k-2}^{j},...,y_{0}^{j}\right) \in f_{j}^{k}(U_{i}\cap U_{j})$
et 
\begin{equation*}
\gamma _{_{ij}}^{k}\left( y_{k-1}^{j},y_{k-2}^{j},...,y_{0}^{j}\right)
=\left( \gamma _{k-1}^{i},\gamma _{k-2}^{i},...,\gamma _{0}^{i}\right) .
\end{equation*}

On a $\cite{DAD}$ les \'{e}galit\'{e}s, 
\begin{equation*}
f_{i}^{k-1}=\theta _{i}^{k-1}\circ f_{i}^{k}\text{ \ et }\gamma
_{_{ij}}^{k-1}\circ \theta _{j}^{k-1}=\theta _{i}^{k-1}\circ \gamma
_{_{ij}}^{k}
\end{equation*}%
c'est \`{a} dire que le diagrame suivant est commutatif

\begin{equation*}
\begin{array}{ccccc}
U_{i}\cap U_{j} & \overset{f_{_{j}}^{k}}{\rightarrow } & f_{j}^{k}(U_{i}\cap
U_{j}) & \overset{\theta _{j}^{k-1}}{\rightarrow } & f_{j}^{k-1}(U_{i}\cap
U_{j}) \\ 
^{Id_{U_{i}\cap U_{j}}}\downarrow &  & \downarrow \gamma _{_{ij}}^{k} &  & 
\downarrow \gamma _{_{ij}}^{k-1} \\ 
U_{i}\cap U_{j} & \overset{f_{i}^{k}}{\rightarrow } & f_{i}^{k}(U_{i}\cap
U_{j}) & \overset{\theta _{i}^{k-1}}{\rightarrow } & f_{i}^{k-1}(U_{i}\cap
U_{j})%
\end{array}%
.
\end{equation*}%
D'o\`{u} l'\'{e}galit\'{e} 
\begin{equation*}
\theta _{i}^{k-1}\left( y_{k-1}^{i},y_{k-2}^{i},...,y_{0}^{i}\right) =\left(
y_{k-2}^{i},...,y_{0}^{i}\right)
\end{equation*}
entraine :%
\begin{eqnarray*}
\ \gamma _{_{ij}}^{k-1}\left( y_{k-2}^{j},...,y_{0}^{j}\right) &=&\gamma
_{_{ij}}^{k-1}\circ \theta _{j}^{k-1}\left(
y_{k-1}^{j},y_{k-2}^{j},...,y_{0}^{j}\right) \\
&=&\theta _{i}^{k-1}\circ \gamma _{_{ij}}^{k}\left(
y_{k-1}^{j},y_{k-2}^{j},...,y_{0}^{j}\right) \\
&=&\theta _{i}^{k-1}\left( \gamma _{k-1}^{i},\gamma _{k-2}^{i},...,\gamma
_{0}^{i}\right) \\
&=&\left( \gamma _{k-2}^{i},...,\gamma _{0}^{i}\right)
\end{eqnarray*}

Ainsi 
\begin{equation*}
\gamma _{_{ij}}^{k}\left( y_{k-1}^{j},y_{k-2}^{j},...,y_{0}^{j}\right)
=\left( \gamma _{k-1}^{i},\gamma _{k-2}^{i},...,\gamma _{0}^{i}\right)
\end{equation*}%
avec%
\begin{equation*}
\left( \gamma _{k-s-1}^{i},\gamma _{k-s-2}^{i},...,\gamma _{0}^{i}\right)
=\gamma _{_{ij}}^{k-s}\left( y_{k-s-1}^{j},...,y_{0}^{j}\right)
\end{equation*}%
pour $s\in \left\{ 1,...,k-1\right\} .$

Il r\'{e}sulte de ce qui pr\'{e}c\`{e}de que 
\begin{equation*}
J_{\gamma _{_{ij}}^{k}}=\left( 
\begin{array}{cccccc}
\frac{\partial \gamma _{k-1}^{i}}{\partial y_{k-1}^{j}} & \frac{\partial
\gamma _{k-1}^{i}}{\partial y_{k-2}^{j}} & ... & ... & ... & \frac{\partial
\gamma _{_{k-1}}^{i}}{\partial y_{0}^{j}} \\ 
0 & \frac{\partial \gamma _{k-2}^{i}}{\partial y_{k-2}^{j}} & \frac{\partial
\gamma _{k-2}^{i}}{\partial y_{k-3}^{j}} & ... & ... & \frac{\partial \gamma
_{_{k-2}}^{i}}{\partial y_{0}^{j}} \\ 
0 & 0 & \frac{\partial \gamma _{k-3}^{i}}{\partial y_{k-3}^{j}} & ... & ...
& \frac{\partial \gamma _{_{k-3}}^{i}}{\partial y_{0}^{j}} \\ 
: & : & 0 & . &  & : \\ 
: & : & : & . & . & : \\ 
0 & 0 & 0 & .... & 0 & \frac{\partial \gamma _{0}^{i}}{\partial y_{0}^{j}}%
\end{array}%
\right) .
\end{equation*}%
Les $\gamma _{_{ij}}^{k}$ sont des isom\'{e}tries pour la m\'{e}trique
quasi-fibr\'{e}e commune $g.$ Par cons\'{e}quent, comme $\left( \frac{%
\partial }{\partial y_{k-1}^{j}},\frac{\partial }{\partial y_{k-2}^{j}},...,%
\frac{\partial }{\partial y_{0}^{j}}\right) $ et $\left( \frac{\partial }{%
\partial y_{k-1}^{i}},\frac{\partial }{\partial y_{k-2}^{i}},...,\frac{%
\partial }{\partial y_{0}^{i}}\right) $ sont des bases orthonorm\'{e}es pour
la m\'{e}trique $g$ alors $J_{\gamma _{_{ij}}^{k}}$ est une matrice
orthogonale. Ce qui entraine que 
\begin{equation*}
J_{\gamma _{_{ij}}^{k}}\circ J_{\gamma _{_{ij}}^{k}}^{t}=I_{k}
\end{equation*}%
o\`{u} $J_{\gamma _{_{ij}}^{k}}^{t}$ est la transpos\'{e}e de $J_{\gamma
_{_{ij}}^{k}}^{t}.$

En utilisant l'\'{e}galit\'{e} 
\begin{equation*}
J_{\gamma _{_{ij}}^{k}}\circ J_{\gamma _{_{ij}}^{k}}^{t}=I_{k},
\end{equation*}
une r\'{e}currence sur $k$ permet de voir que $J_{\gamma
_{_{ij}}^{k}}=\left( \varepsilon _{ijrs}^{k}\right) _{rs}$ avec 
\begin{equation*}
\varepsilon _{ijrs}^{k}=\{%
\begin{array}{c}
\pm 1\text{ si }r=s \\ 
0\text{ si }r\neq s%
\end{array}%
.
\end{equation*}%
Ainsi, chaque feuilletage riemanien $\mathcal{F}_{k}$ est transversalement
diagonal.
\end{proof}

\section{Extension de feuilletages riemanniens minimals sur une vari\'{e}t%
\'{e} compacte connexe}

Le r\'{e}sultat qui suit est une g\'{e}n\'{e}ralisation partielle du th\'{e}%
or\`{e}me 2.6 se trouvant dans ($\cite{DAD},\cite{DD}).$

\begin{theorem}
Soit $\mathcal{G}$ l'alg\`{e}bre de Lie structurale d'un feuilletage
riemannien minimal $\left( M,\mathcal{F},g_{_{T}}\right) $ sur une vari\'{e}t%
\'{e} compacte $M$ et $g$ une m\'{e}trique $\mathcal{F}$-quasifibr\'{e}e sur 
$M$ de m\'{e}trique $\mathcal{F-}$transverse $g_{_{T}}.$

Alors:

i) Il existe une sous-alg\`{e}bre de Lie $\mathcal{G}_{_{\mathcal{F}}}$ de $%
\mathcal{G}$ telle qu'\`{a} toute extension $\mathcal{F}^{\prime }$ de $%
\mathcal{F}$ correspond (de fa\c{c}on unique) une sous-alg\`{e}bre de Lie $%
\mathcal{G}^{\prime }$ telle que $\mathcal{G}_{_{\mathcal{F}}}\subset 
\mathcal{G}^{\prime }$.

ii) Toute extension $\mathcal{F}^{\prime }$ de $\mathcal{F}$ est un
feuilletage riemannien et $g$ est une m\'{e}trique $\mathcal{F}^{\prime }$%
-quasifibr\'{e}e.
\end{theorem}

\begin{proof}
i) Soit $\overline{F^{\natural }}$ l'adh\'{e}rence d'une feuille $%
F^{\natural }$ du feuilletage relev\'{e} $\mathcal{F}^{\natural }$ de $%
\mathcal{F}$ dans le fibr\'{e} des rep\`{e}res orthonorm\'{e}s transverses $%
M^{\natural }$ de $M$ et $\mathcal{F}_{_{\overline{F^{\natural }}%
}}^{\natural }$ la restriction $\mathcal{F}^{\natural }$ \`{a} $\overline{%
F^{\natural }}.$

Comme $\mathcal{F}$ est \`{a} feuille denses alors $\pi :\overline{%
F^{\natural }}\rightarrow M$ est une fibration principal.

Consid\'{e}rons une extension $\mathcal{F}^{\prime }$ de $\mathcal{F}$ et
les feuilletages $\pi ^{\ast }\mathcal{F}$ et $\pi ^{\ast }\mathcal{F}%
^{\prime }$ images r\'{e}ciproques de $\mathcal{F}$ et de $\mathcal{F}%
^{\prime }$.

On v\'{e}rifie aisement que $\mathcal{F}_{_{\overline{F^{\natural }}%
}}^{\natural }\subset \pi ^{\ast }\mathcal{F}\subset $ $\pi ^{\ast }\mathcal{%
F}^{\prime }.$ Il en r\'{e}sulte (\textit{cf. th\'{e}o. 2.6}) qu'il existe
deux sous-alg\`{e}bres de Lie $\mathcal{G}_{_{\mathcal{F}}}$ et $\mathcal{G}%
^{\prime }$ de $\mathcal{G}$ v\'{e}rifiant $\mathcal{G}_{_{\mathcal{F}%
}}\subset \mathcal{G}^{\prime }$ \ et correspondant respectivement \`{a} $%
\pi ^{\ast }\mathcal{F}$ et $\pi ^{\ast }\mathcal{F}^{\prime }.$

On obtient ainsi la sous-alg\`{e}bre de Lie $\mathcal{G}^{\prime }$ de $%
\mathcal{G}$ correspondant $\mathcal{F}^{\prime }.$

ii) Soit $U^{\natural }$ un ouvert distingu\'{e} \`{a} la fois pour $%
\mathcal{F}_{_{\overline{F^{\natural }}}}^{\natural }$ et $\pi ^{\ast }%
\mathcal{F}^{\prime }$ de sorte que $\pi \left( U^{\natural }\right) $ soit
un ouvert \`{a} la fois distingu\'{e} pour $\mathcal{F}$ et $\mathcal{F}%
^{\prime }.$

Consid\'{e}rons $f^{\natural },$ $f^{\prime \natural }$, $f$ et $f^{\prime }$
les projections respectives associ\'{e}es respectivement aux feuilletages
simples $\left( U^{\natural },\mathcal{F}_{_{\overline{F^{\natural }}%
}}^{\natural }\right) ,$ $\left( U^{\natural },\pi ^{\ast }\mathcal{F}%
^{\prime }\right) ,$ $\left( \pi \left( U^{\natural }\right) ,\mathcal{F}%
\right) $ et $\left( \pi \left( U^{\natural }\right) ,\mathcal{F}^{\prime
}\right) $.

Il existe une submersion $\widetilde{\pi }:f^{\natural }\left( U^{\natural
}\right) \rightarrow f\left( \pi \left( U^{\natural }\right) \right) $
rendant le diagramme 
\begin{equation*}
\begin{array}{ccc}
U^{\natural } & \overset{\pi }{\rightarrow } & \pi \left( U^{\natural
}\right) \\ 
\downarrow f^{\natural } &  & \downarrow f \\ 
f^{\natural }\left( U^{\natural }\right) & \overset{\widetilde{\pi }}{%
\rightarrow } & f\left( \pi \left( U^{\natural }\right) \right)%
\end{array}%
\end{equation*}%
commutatif.

Soit $\theta ^{\natural }$ la liaison entre $\left( U^{\natural },\mathcal{F}%
_{_{\overline{F^{\natural }}}}^{\natural }\right) $ et $\left( U^{\natural
},\pi ^{\ast }\mathcal{F}^{\prime }\right) $ et $\theta $ la liaison entre $%
\left( \pi \left( U^{\natural }\right) ,\mathcal{F}\right) $ et $\left( \pi
\left( U^{\natural }\right) ,\mathcal{F}^{\prime }\right) .$

On a le diagramme suivant qui est commutatif :%
\begin{equation*}
\begin{array}{ccccccc}
&  & U^{\natural } & \overset{\pi }{\rightarrow } & \pi \left( U^{\natural
}\right) &  &  \\ 
& ^{f_{{}}^{\prime \natural }}\swarrow ^{{}} & \downarrow f^{\natural } &  & 
\downarrow f & \searrow ^{f^{\prime }} &  \\ 
\theta ^{\natural }\left( f^{\natural }\left( U^{\natural }\right) \right) & 
\overset{\theta ^{\natural }}{\longleftarrow } & f^{\natural }\left(
U^{\natural }\right) & \overset{\widetilde{\pi }}{\rightarrow } & f\left(
\pi \left( U^{\natural }\right) \right) & \overset{\theta }{\rightarrow } & 
f^{\prime }\left( \pi \left( U^{\natural }\right) \right)%
\end{array}%
\text{ \ \ }\left( \ast \right) \text{\ }.
\end{equation*}%
Par constuction de $\pi ^{\ast }\mathcal{F}^{\prime },$ $\pi $ envoie une
feuille de $\pi ^{\ast }\mathcal{F}^{\prime }$ \ sur une feuille de $%
\mathcal{F}^{\prime }.$ De ce fait $\pi $ envoie les fibres de $f^{\prime
\natural }$ sur les fibres de $f^{\prime }.$

Soit donc $x^{\natural }\in \theta ^{\natural }\left( f^{\natural }\left(
U^{\natural }\right) \right) ,$ il existe $x\in f^{\prime }\left( \pi \left(
U^{\natural }\right) \right) $ tel que 
\begin{equation*}
\pi \left( \left( f^{\prime \natural }\right) ^{-1}\left( x^{\natural
}\right) \right) =\left( f^{\prime }\right) ^{-1}\left( x\right) .
\end{equation*}

Or d'apr\`{e}s le diagramme $\left( \ast \right) $ 
\begin{equation*}
\left( \theta ^{\natural }\right) ^{-1}\left( x^{\natural }\right) =\
f^{\natural }\left( \left( f^{\prime \natural }\right) ^{-1}\left(
x^{\natural }\right) \right) \text{ \ et \ }\left( \theta \right)
^{-1}\left( x\right) =\ f\left( \left( f^{\prime }\right) ^{-1}\left(
x\right) \right)
\end{equation*}%
donc%
\begin{eqnarray*}
\widetilde{\pi }\left( \left( \theta ^{\natural }\right) ^{-1}\left(
x^{\natural }\right) \right) &=&\widetilde{\pi }\circ \ f^{\natural }\left(
\left( f^{\prime \natural }\right) ^{-1}\left( x^{\natural }\right) \right)
\\
&=&f\circ \pi \left( \left( f^{\prime \natural }\right) ^{-1}\left(
x^{\natural }\right) \right) \\
&=&f\circ \left( \left( f^{\prime }\right) ^{-1}\left( x\right) \right) \\
&=&\left( \theta \right) ^{-1}\left( x\right)
\end{eqnarray*}%
Il en r\'{e}sulte que $\widetilde{\pi }$ projette les fibres de $\theta
^{\natural }$ sur les fibres de $\theta .$

Ceci \'{e}tant, pour montrer que $\left( \pi \left( U^{\natural }\right) ,%
\mathcal{F}^{\prime }\right) $ est un feuilletage riemannien on montre dans $%
\cite{DAD}$ qu'il suffit de montrer que la liaison $\theta $ est une
submersion riemannienne.

Ceci dit, consid\'{e}rons le faisceau transverse central $\mathcal{C(M}$,$%
\mathcal{F}_{_{\overline{F^{\natural }}}}^{\natural }\mathcal{)}$ du
feuilletage de Lie minimal $\mathcal{F}_{_{\overline{F^{\natural }}%
}}^{\natural }.$

On sait que $\mathcal{C(M}$,$\mathcal{F}_{_{\overline{F^{\natural }}%
}}^{\natural }\mathcal{)}$ est localement trivial et s'identifie aux germes d%
\'{e}finies par l'alg\`{e}bre de Lie structurale $\ell (\mathcal{M},\mathcal{%
F}_{_{\overline{F^{\natural }}}}^{\natural }\mathcal{)}$de $\mathcal{F}_{_{%
\overline{F^{\natural }}}}^{\natural }$.

Comme $\mathcal{C(M}$,$\mathcal{F}_{_{\overline{F^{\natural }}}}^{\natural }%
\mathcal{)}$ est localement trivial de fibre type $\mathcal{G}$ et comme $%
\mathcal{G}^{\prime }$\ est une sous-alg\`{e}bre de Lie de $\mathcal{G}$
alors on peut consid\'{e}rer le \textquotedblright
sous-faisceau\textquotedblright\ $\mathcal{C}_{_{\mathcal{G}^{\prime }}}%
\mathcal{(M}$,$\mathcal{F}_{_{\overline{F^{\natural }}}}^{\natural }\mathcal{%
)}$\ de $\mathcal{C(M}$,$\mathcal{F}_{_{\overline{F^{\natural }}}}^{\natural
}\mathcal{)}$ de fibre type la sous-alg\`{e}bre de Lie $\mathcal{G}^{\prime
} $ de $\mathcal{G}$ correspondant \`{a} $\pi ^{\ast }\mathcal{F}^{\prime }.$

On sait que $\mathcal{C}_{_{\mathcal{G}^{\prime }}}\mathcal{(M}$,$\mathcal{F}%
_{_{\overline{F^{\natural }}}}^{\natural }\mathcal{)}$ d\'{e}finit , par ses
orbites dans $\overline{F^{\natural }}$, le feuilletage $\mathcal{F}^{\prime
}$. Ainsi, sous la supposition que l'ouvert $U^{\natural }$ distingu\'{e} 
\`{a} la fois pour $\mathcal{F}_{_{\overline{F^{\natural }}}}^{\natural }$
et $\pi ^{\ast }\mathcal{F}^{\prime }$ est aussi un ouvert de trivialisation
local du faisceau $\mathcal{C(M}$,$\mathcal{F}_{_{\overline{F^{\natural }}%
}}^{\natural }\mathcal{)}$, on peut affirmer que les fibres de la liaison $%
\theta ^{\natural }$ sont les orbites d'un faisceau de germes de champs de
Killing $\mathcal{F}_{_{\overline{F^{\natural }}}}^{\natural }-$transverses.

On montre dans $\cite{MOL}$ que le sous-faisceau $\mathcal{C}_{_{\mathcal{G}%
^{\prime }}}\mathcal{(M}$,$\mathcal{F}_{_{\overline{F^{\natural }}%
}}^{\natural }\mathcal{)}$ de germes de champs de Killing $\mathcal{F}_{_{%
\overline{F^{\natural }}}}^{\natural }-$transverses se proj\`{e}tent par la
fibration $\pi :\overline{F^{\natural }}\rightarrow M$ en un sous-faisceau $%
\mathcal{C}_{_{\mathcal{G}^{\prime }}}\mathcal{(M}$,$\mathcal{F)}$ de germes
de champs de killing $\mathcal{F-}$transverses. Il r\'{e}sullte de ce qui pr%
\'{e}c\`{e}de et du diagramme commutatif $\left( \mathcal{\ast }\right) $
que les fibres de\ la liaison $\theta $\ entre $\left( \pi \left(
U^{\natural }\right) ,\mathcal{F}\right) $ et $\left( \pi \left( U^{\natural
}\right) ,\mathcal{F}^{\prime }\right) $ sont les orbites d'un faisceau de
germes de champs de Killing $\mathcal{F}-$transverses. Ce qui signifie que
le feuilletage $\mathcal{F}_{\theta }$ d\'{e}fini par $\theta $ est un
feuilletage riemannien. En d'autres termes $\theta $ est une submersion
riemannienne.

Il r\'{e}sulte de ce qui pr\'{e}c\`{e}de que $\mathcal{F}^{\prime }$ est un
feuilletage riemannien.

On v\'{e}rifie aisement, sous la supposition que l'ouvert $U^{\natural }$
distingu\'{e} \`{a} la fois pour $\mathcal{F}_{_{\overline{F^{\natural }}%
}}^{\natural }$ et $\pi ^{\ast }\mathcal{F}^{\prime }$ est aussi un ouvert
de trivialisation local du faisceau $\mathcal{C(M}$,$\mathcal{F}_{_{%
\overline{F^{\natural }}}}^{\natural }\mathcal{)}$, que l'ouvert $U=\pi
\left( U^{\natural }\right) $ \`{a} la fois distingu\'{e} pour $\mathcal{F}$
et $\mathcal{F}^{\prime }$ est aussi un ouvert de trivialisation local du
sous-faisceau $\mathcal{C}_{_{\mathcal{G}^{\prime }}}\mathcal{(M}$,$\mathcal{%
F)}$.

Ceci dit, on suppose que l'ouvert $U=\pi \left( U^{\natural }\right) $ est
un ouvert de trivialisation local du sous-faisceau $\mathcal{C}_{_{\mathcal{G%
}^{\prime }}}\mathcal{(M}$,$\mathcal{F)}$.

Soit $\mathcal{C}_{_{\mathcal{G}^{\prime }}}^{U}\mathcal{(M}$,$\mathcal{F)}$
la restriction \`{a} $U$ du sous-faisceau $\mathcal{C}_{_{\mathcal{G}%
^{\prime }}}\mathcal{(M}$,$\mathcal{F)}$, $\mathcal{F}_{_{U}}$ et $\mathcal{F%
}_{_{U}}^{\prime }$ les restrictions respectives de $\mathcal{F}$ et $%
\mathcal{F}^{\prime }$ \`{a} $U,$ $\mathcal{X}\left( \mathcal{F}%
_{_{U}}\right) $ et $\mathcal{X}\left( \mathcal{F}_{_{U}}^{\prime }\right) $%
\ les $\mathcal{A}^{0}\left( U\right) -$modules des champs de vecteurs
tangents respectivement \`{a} $\mathcal{F}_{_{U}}$ et $\mathcal{F}%
_{_{U}}^{\prime }$ , $Y$ et$\ Z$ deux champs de vecteurs $\mathcal{F}%
_{_{U}}^{\prime }$-feuillet\'{e}s sur $U$ normaux \`{a} $\mathcal{F}%
_{_{U}}^{\prime }$ et $X_{_{_{\mathcal{F}_{_{U}}^{\prime }}}}\in \mathcal{X}%
\left( \mathcal{F}_{_{U}}^{\prime }\right) .$

On a $\mathcal{C}_{_{\mathcal{G}^{\prime }}}^{U}\mathcal{(M}$,$\mathcal{F)}$
qui est trivial de fibre type $\mathcal{\pi }_{\ast }\left( \mathcal{G}%
^{\prime }\right) $ et il d\'{e}finit $\mathcal{F}_{_{U}}^{\prime }$
transversalement \`{a} $\mathcal{F}_{_{U}}.$ En fait $\mathcal{F}%
_{_{U}}^{\prime }$ est d\'{e}fini par les orbites $\cite{AM}$ de l'alg\`{e}%
bre de Lie $\mathcal{\pi }_{\ast }\left( \mathcal{G}^{\prime }\right) $ de
champs killings $\mathcal{F}-$transverses locals. Par cons\'{e}quent 
\begin{equation*}
\mathcal{X}\left( \mathcal{F}_{_{U}}^{\prime }\right) =\mathcal{X}\left( 
\mathcal{F}_{_{U}}\right) \mathcal{\oplus }\left( \mathcal{\mathcal{A}}%
^{0}\left( U\right) \mathcal{\otimes \pi }_{\ast }\left( \mathcal{G}^{\prime
}\right) \right) .
\end{equation*}

Comme 
\begin{equation*}
\mathcal{X}\left( \mathcal{F}_{_{U}}^{\prime }\right) =\mathcal{X}\left( 
\mathcal{F}_{_{U}}\right) \mathcal{\oplus }\left( \mathcal{\mathcal{A}}%
^{0}\left( U\right) \mathcal{\otimes \pi }_{\ast }\left( \mathcal{G}^{\prime
}\right) \right) ,
\end{equation*}
il existe $X_{_{\mathcal{F}_{_{U}}}}\in \mathcal{X}\left( \mathcal{F}%
_{_{U}}\right) $, $X_{_{\mathcal{G}^{\prime }}}\in \mathcal{\pi }_{\ast
}\left( \mathcal{G}^{\prime }\right) $ et $f\in \mathcal{\mathcal{A}}%
^{0}\left( U\right) $ tels que 
\begin{equation*}
X_{_{_{\mathcal{F}_{_{U}}^{\prime }}}}=X_{_{\mathcal{F}_{_{U}}}}+fX_{_{%
\mathcal{G}^{\prime }}}.
\end{equation*}%
Ainsi 
\begin{equation*}
X_{_{_{\mathcal{F}_{_{U}}^{\prime }}}}g\left( Y,Z\right) =X_{_{\mathcal{F}%
_{_{U}}}}g\left( Y,Z\right) +fX_{_{\mathcal{G}^{\prime }}}g\left( Y,Z\right)
.
\end{equation*}

Comme $\mathcal{F}_{_{U}}\subset \mathcal{F}_{_{U}}^{\prime }$ et les champs
de vecteurs $Y$ et$\ Z$ normaux \`{a} $\mathcal{F}_{_{U}}^{\prime }$ alors $%
Y $ et$\ Z$ sont aussi normaux \`{a} $\mathcal{F}_{_{U}}.$ Or la m\'{e}%
trique $g$ est $\mathcal{F}$-quasifibr\'{e}e donc 
\begin{equation*}
X_{_{\mathcal{F}_{_{U}}}}g\left( Y,Z\right) =g\left( \left[ X_{_{\mathcal{F}%
_{_{U}}}},Y\right] ,Z\right) +g\left( Y,\left[ X_{_{\mathcal{F}_{_{U}}}},Z%
\right] \right) .
\end{equation*}%
Les champs de vecteurs $Y$ et$\ Z$ \'{e}tant $\mathcal{F}_{_{U}}^{\prime }$%
-feuillet\'{e}s on a $\left[ X_{_{\mathcal{F}_{_{U}}}},Y\right] \in \mathcal{%
X}\left( \mathcal{F}_{_{U}}^{\prime }\right) $ et $\left[ X_{_{\mathcal{F}%
_{_{U}}}},Z\right] \in \mathcal{X}\left( \mathcal{F}_{_{U}}^{\prime }\right)
.$ Ceci entraine, $Y$ et$\ Z$ \'{e}tant normaux \`{a} $\mathcal{F}%
_{_{U}}^{\prime },$ que 
\begin{equation*}
X_{_{\mathcal{F}_{_{U}}}}g\left( Y,Z\right) =g\left( \left[ X_{_{\mathcal{F}%
_{_{U}}}},Y\right] ,Z\right) +g\left( Y,\left[ X_{_{\mathcal{F}_{_{U}}}},Z%
\right] \right) =0.
\end{equation*}

D'autre part $X_{_{\mathcal{G}^{\prime }}}$\ est un champ de Killing
transverse local pour $\mathcal{F}$. D'o\`{u} 
\begin{equation*}
X_{_{\mathcal{G}^{\prime }}}g\left( Y,Z\right) =g\left( \left[ X_{_{\mathcal{%
G}^{\prime }}},Y\right] ,Z\right) +g\left( Y,\left[ X_{_{\mathcal{G}^{\prime
}}},Z\right] \right) .
\end{equation*}%
Pour les m\^{e}mes raisons que pr\'{e}c\'{e}demment on obtient 
\begin{equation*}
X_{_{\mathcal{G}^{\prime }}}g\left( Y,Z\right) =g\left( \left[ X_{_{\mathcal{%
G}^{\prime }}},Y\right] ,Z\right) +g\left( Y,\left[ X_{_{\mathcal{G}^{\prime
}}},Z\right] \right) =0.
\end{equation*}

Il r\'{e}sulte de ce qui pr\'{e}c\`{e}de que $X_{_{_{\mathcal{F}%
_{_{U}}^{\prime }}}}g\left( Y,Z\right) =0.$ Et, ceci signifie que la m\'{e}%
trique $g$ est $\mathcal{F}^{\prime }-$quasi-fibr\'{e}e.
\end{proof}

Les r\'{e}sultats de la proposition 3.3 et des th\'{e}or\`{e}mes 3.7 et 4.1
permettent d'\'{e}tablir aisement ce qui suit:

\begin{corollary}
Soit $\mathcal{F}$ est un feuilletage riemannien minimal sur une vari\'{e}t%
\'{e} compacte connexe.

Alors $\mathcal{F}$ admet un drapeau complet d'extension si et seulement si $%
\mathcal{F}$ est transversalement diagonal.
\end{corollary}

\end{document}